\documentclass[12pt]{article}
\usepackage{latexsym}
\usepackage{amsfonts}
\usepackage{amsmath}
\usepackage{amsthm}
\usepackage{mathrsfs}
\usepackage[pdftex]{graphicx}
\usepackage{amssymb}
\usepackage{wasysym}
\usepackage{enumerate}
\def\ZZ{\mathbb Z}
\def\RR{\mathbb R}
\def\CC{\mathbb C}

\def\CP{\mathbb{CP}}
\def\eea{\end{eqnarray*}}

\DeclareMathAlphabet{\zap}{OT1}{pzc}{m}{it}
\DeclareMathOperator{\grad}{grad}

\DeclareMathOperator{\Diff}{\zap{Diff}_H}

\newtheorem{main}{Theorem}

\newtheorem{defn}{Definition}
\newtheorem{thm}{Theorem}
\newtheorem*{them}{Theorem}
\newtheorem{prop}{Proposition}

\newtheorem{lem}{Lemma}

\usepackage[OT2,OT1]{fontenc}
\def\cyr{%
\renewcommand\rmdefault{wncyr}%
\renewcommand\sfdefault{wncyss}%
\renewcommand\encodingdefault{OT2}%
\normalfont
\selectfont}
\DeclareTextFontCommand{\textcyr}{\cyr}

\newcommand{\rad}{\text{\cyr   ya}}

\begin{document}
\sloppy

\title{The Einstein-Maxwell Equations and\\ Conformally K\"ahler Geometry}

\author{Claude LeBrun\thanks{Supported 
in part by  NSF grant DMS-1205953.} 
\\ 
SUNY Stony
 Brook 
  }

\date{April 23,  2015}
\maketitle

\bigskip

\begin{abstract}  Page's Einstein metric on $\CP_2 \# \overline{\CP}_2$ is  conformally related to  an extremal 
K\"ahler metric. Here we construct a  family of conformally K\"ahler solutions of the Einstein-Maxwell equations that
deforms the Page metric, while  sweeping out the entire K\"ahler cone of $\CP_2 \# \overline{\CP}_2$.
The same method also yields analogous  solutions  on 
every Hirzebruch surface. This allows us to display infinitely many geometrically distinct families of 
solutions of the Einstein-Maxwell equations on the smooth $4$-manifolds $S^2\times S^2$ and $\CP_2 \# \overline{\CP}_2$. 
\end{abstract}

Let $(M,h)$ be a connected, oriented Riemannian $4$-manifold. We will say that $h$ is an {\em Einstein-Maxwell metric} if there is 
a  $2$-form $F$ on $M$ such that the pair  $(h,F)$ satisfies the 
{\em Einstein-Maxwell equations}
\begin{eqnarray}
dF&=&0 \label{closed}\\
d \star F&=&0 \label{coclosed}\\
\Big[r+ F\circ F\Big]_0
&=&0 \label{energy}
\end{eqnarray}
where $r$ is the Ricci tensor of $h$,
the subscript 
$[~]_0$ indicates
the trace-free part with respect to $h$, 
and  the symmetic tensor 
$(F\circ F)_{jk}= {F_j}^\ell F_{\ell k}$ is obtained by
composing $F$ with itself as an endomorphism of $TM$. 
In physics terminology, equations  (\ref{closed}--\ref{coclosed})  are sometimes
 called the {\em Euclidean Einstein-Maxwell equations with cosmological constant}.
 This terminology  emphasizes two important points: 
 we are taking $h$ to  be a Riemannian metric rather than a Lorentzian one;
 and, while these equations   imply that the scalar curvature $s$ of $h$ must be  constant, this 
 constant is allowed to  be non-zero.

These equations turn out to  naturally arise in connection with many interesting   geometric questions, 
including  some of the most active current research topics
 in  K\"ahler geometry. For example, if 
$M$ admits a complex structure, and if $h$ is a 
constant-scalar-curvature K\"ahler (cscK) metric on $M$, then  $h$ is Einstein-Maxwell. Indeed, if we set  $F= \frac{2-s}{4}\omega + \rho$, where $\omega$ and $\rho$ are
respectively the K\"ahler and Ricci forms of the K\"ahler metric $h$, then $(r,F)$ solves the Einstein-Maxwell equations.

The existence of cscK metrics for a fixed complex structure and fixed K\"ahler class is actually a difficult open problem, and is the subject of a great deal of current cutting-edge research \cite{dontor}. However, the problem becomes much more  
tractable \cite{arpa2,klp,yujen}  
if one instead  just asks whether or not there is {\em some} complex structure 
in a given deformation class and {\em some} compatible K\"ahler class for which a solution exists. 
Using this observation 
 in tandem another recent development \cite{chenlebweb} in K\"ahler geometry, it is then relatively easy to prove  the following  \cite{lebem}:

\begin{them}
Let $M$ be the underlying smooth compact $4$-manifold of a compact complex surface. If $M$ is of K\"ahler type --- i.e. if 
$b_1(M)$ is even --- then $M$ admits Einstein-Maxwell metrics. By contrast, if $M$ is {\em not} of K\"ahler type and has 
vanishing geometric genus, then $M$ does {\em not} admit Einstein-Maxwell metrics. 
\end{them}

This surprising relationship between Einstein-Maxwell  metrics and the K\"ahler condition immediately raises the following question: 
If $M$ is  the underlying smooth $4$-manifold of a compact complex surface, is every 
Einstein-Maxwell metric on $M$ actually a K\"ahler metric? However, the answer turns  out to  be {\em no}. In fact, the proof of the above theorem 
depends in part on the fact that  the one-point 
blow-up of  the complex projective plane admits a non-K\"ahler  Einstein metric discovered by Page \cite{page}. On the other hand, as pointed out by 
Derdzi\'nski \cite{derd}, the Page metric, while not K\"ahler, is nonetheless {\em conformal} to a K\"ahler metric. Are there other Einstein-Maxwell
metrics on this same space which are conformally K\"ahler? The answer is {\em yes}!

\begin{main} \label{premier} 
Let $M\approx \CP_2 \#\overline{\CP}_2$ be the  blow-up of the complex projective plane at  a point, equipped with its standard complex structure, and 
let $\Omega$ be any K\"ahler class on $M$. 
Then $\Omega$ contains a K\"ahler metric $g$ which is 
 conformal to a (non-K\"ahler) Einstein-Maxwell  metric  $h$. \end{main}
 
 However, in contrast to the situation for cscK metrics \cite{donaldsonk1}, conformally Einstein-Maxwell metrics are generally not uniquely determined by 
 their K\"ahler classes up to complex automorphisms. Indeed, the present author has elsewhere \cite{lebem14} shown that this the non-uniqueness
 phenomenon occurs on $M=\CP_1\times\CP_1$, where certain K\"ahler class contain both a cscK metric and a K\"ahler metric of non-constant scalar curvature that  is
 conformally Einstein-Maxwell. For the one-point blow-up of the complex projective plane, the situation is analogous: 
 
\begin{main} \label{deuxieme}  The metric $g$  in Theorem \ref{premier} is not always unique.
 To make this more precise, express an arbitrary K\"ahler class as   $$\Omega = u {\mathcal L} - v{\zap E}$$
where ${\mathcal L}$ and ${\zap E}$ are respectively the Poincar\'e duals of a   projective line and the exceptional curve,
and where $u$ and $v$ are  real numbers  with $u> v > 0$. 
If $u/v >  9$, 
then $\Omega$ contains three geometrically distinct, ${\mathbf U}(2)$-invariant  K\"ahler metrics $g$ which are conformal to  Einstein-Maxwell metrics $h$; however,  two 
of the resulting 
Einstein-Maxwell metrics $h$ are actually isometric,  in an orientation-reversing manner. 
By contrast, when $u/v\leq 9$, there is, up to complex automorphisms,   a unique   K\"ahler metric $g$ in $\Omega$ which is conformal to an Einstein-Maxwell metric $h$ 
whose isometry group contains ${\mathbf U}(2)$. 
\end{main}

We emphasize that our  proof of the uniqueness assertion in Theorem \ref{deuxieme}  is entirely  dependent  on the  assumption of ${\mathbf U}(2)$-invariance. 
An intriguing problem, which we leave for the interested 
reader, is to determine if any such unicity  persists in the absence of this assumption. However,  symmetry assumptions certainly could  
play a decisive role here. For example, exactly one of the K\"ahler metrics $g$ which  we will construct  in each K\"ahler class
$\Omega$ engenders an Einstein-Maxwell metric $h$ which has  an orientation-reversing isometry
in addition to  its ${\mathbf U}(2)$ symmetry. 
If we had required  $h$ to  also have such an isometry from the outset, the 
  bifurcation phenomenon described by Theorem \ref{deuxieme} would therefore have been eliminated, and 
we would then be left with  exactly one solution $g$ in every K\"ahler class.

Of course, our definition of an Einstein-Maxwell metric allows for the possibility that the metric might actually be Einstein, corresponding to the possibility that the 
$2$-form might vanish. This does indeed occur on $\CP_2\# \overline{\CP}_2$, as the Page metric is certainly an example. On the other hand, 
the other  solutions under discussion here are definitely  {\em not} Einstein:

\begin{main}\label{pages}
There is  a unique  
  value  of $u/v$, given by 
  \begin{eqnarray*} 
\frac{u}{v}&=& \left[\frac{1}{2} \left(\sqrt[3]{1 + \sqrt{2}}-\frac{1}{\sqrt[3]{1 + \sqrt{2}}}\right)\right]^{-1/2}  +  \\&&
 2\sqrt{\left[\frac{1}{2} \left(\sqrt[3]{1 + \sqrt{2}}-\frac{1}{\sqrt[3]{1 + \sqrt{2}}}\right)\right]^{1/2} -\left[\frac{1}{2} \left(\sqrt[3]{1 + \sqrt{2}}-\frac{1}{\sqrt[3]{1 + \sqrt{2}}}\right)\right]^{2} } 
 \\&\approx& 3.18393 34,  \end{eqnarray*}
   for which the Einstein-Maxwell metric of Theorem \ref{deuxieme}  becomes  a constant times Page's Einstein metric {\em \cite{page}}. 
  For other values of $u/v$, these Einstein-Maxwell metrics are not Bach-flat, and so are not even conformally Einstein. 
\end{main}

The same framework  used to prove the above results also produces  solutions on
every {\em Hirzebruch surface}. Recall \cite{bpv} that  a Hirzebruch surface is a compact
complex surface which is a holomorphic $\CP_1$-bundle over $\CP_1$. Every  Hirzebruch can be expressed as 
$$\Sigma_k:= \mathbb{P} ({\mathcal O}\oplus {\mathcal O}(k))$$
for a unique   non-negative integer $k$,  where  $\mathbb{P}$ indicates the fiber-wise projectivization of 
a holomorphic rank-$2$, and, by a standard abuse of notation, 
${\mathcal O}$ and ${\mathcal O}(k)$ respectively denotes the trivial line bundle and the holomorphic line bundle of degree $k$ over $\CP_1$. The $\Sigma_k$ are mutually 
non-isomorphic  as complex manifolds, 
but there are only two diffeomorphism types: 
$$\Sigma_k \approx  \begin{cases}
  \CP_2 \# \overline{\CP}_2  ,  &\text{if $k$ is odd; or}\\
S^2 \times S^2, &\text{if $k$ is even.}
   \end{cases}
 $$
In fact, up to biholomorphism,  the Hirzebruch surfaces  are the {\em only} complex surfaces diffeomorphic to $\CP_2 \# \overline{\CP}_2$
or $S^2 \times S^2$. 

For the Einstein-Maxwell metrics considered here, the behavior observed   
 on most  Hirzebruch surfaces is simpler than that seen   on $\Sigma_0= \CP_1\times\CP_1$ or on 
 the one-point blow-up $\Sigma_1$ 
 of $\CP_2$: 

\begin{main} 
\label{quatrieme} 
Let $M=\Sigma_k$ be the $k^{\rm th}$ Hirzebruch surface, with its fixed complex structure, and let 
$\Omega$ be any K\"ahler class on $M$. Then $\Omega$ contains a K\"ahler metric $g$ which is conformal to an
Einstein-Maxwell metric $h$. Moreover, if $k\geq 2$, there is a unique such  K\"ahler metric $g$ which is invariant under the standard action of 
 ${\mathbf U}(2)$
on $\Sigma_k$. 
\end{main}

However, every Hirzebruch surface is diffeomorphic to either $S^2 \times S^2$ or $\CP_2\# \overline{\CP}_2$, so 
Theorem \ref{quatrieme}  asserts the existence of an infinite number of families of solutions on both of these smooth compact $4$-manifolds. 
It seems plausible   that  these families may  actually belong to different connected components of the moduli space of solutions. 
In any case, our construction certainly does imply an interesting result  in this direction:

\begin{main}
\label{cinquieme} Let the smooth oriented $4$-manifold  $M$ be  either $\CP_2 \#\overline{\CP}_2$  or $S^2\times S^2$, and, 
for any $\Omega \in H^2(M, \RR)$ with $\Omega^2 > 0$, let 
$$\mathscr{M}_\Omega =\{ \mbox{solutions $(h,F)$ of (\ref{closed}--\ref{energy}) on $M$}~|~ F^+\in \Omega\}/[ \Diff(M)\times \RR^+]$$
be the moduli space of $\Omega$-compatible solutions of the Einstein-Maxwell equations on $M$; here $\Diff (M)$ denotes the 
group of diffeomorphisms of $M$ which act trivially on $H^2(M)$, and $\RR^+$ acts by rescaling $h$, without changing $F$. Then, for every positive integer  $\mathbf{N}$,
there is a choice of  $\Omega$ such that $\mathscr{M}_\Omega$ has at least $\mathbf{N}$ connected components. 
\end{main}

\section{Generalities} \label{prelim} 

While the   physical  interest of the 
Einstein-Maxwell equations may seem   self-evident, 
  these equations are also inherently 
interesting for reasons that are intrinsic  to  Riemannian geometry. For example, there 
are several remarkable scalar curvature estimates \cite{gl1,lno,lebem} in $4$-dimensional Riemannian geometry 
that depend on the cohomology class of a harmonic self-dual $2$-form. 
Such estimates typically amount to assertions    about the volume-normalized Einstein-Hilbert 
functional 
$$h ~\stackrel{\mathfrak{S}}{\longmapsto}~  V_h^{-1/2} \int_M s_hd\mu_h$$
on the space $\mathscr{G}$ of smooth Riemannian metrics on $M$, where $s_h$ denotes the scalar curvature of $h$,  $d\mu_h$ is the metric volume measure, and 
$V_h$ is the total volume of $(M,h)$. 
If $M$ is a smooth compact oriented $4$-manifold, and if $\Omega \in H^2(M,\RR)$ is a fixed cohomology class with $\Omega ^2 > 0$, 
 it is therefore natural to consider the set 
$\mathscr{G}_\Omega$ of smooth Riemannian metrics $h$ on $M$ for which  the harmonic representative $\omega$ of $\Omega$ is self-dual. 
One can then show that  $\mathscr{G}_\Omega$  is a  Fr\'echet manifold, and indeed  is 
a closed submanifold of the space of smooth Riemannian metrics, of finite codimension $b_-(M)$. 
This allows us to consider the variational problem arising from the 
restriction 
\begin{eqnarray}
\label{hilbert}
{\mathfrak S}|_{\mathscr{G}_\Omega}: \mathscr{G}_\Omega &\longrightarrow&\RR \\
h&\longmapsto& \frac{\int_M s_hd\mu_h}{\sqrt{\int_M d\mu_h}} \nonumber
\end{eqnarray}
of the normalized Einstein-Hilbert functional to the $\Omega$-adapted  metrics. 
The critical points of this variational problems are then \cite{lebem, lebem14} just the solutions of the  Einstein-Maxwell equations
which are appropriately related to $\Omega$: 

\begin{prop} A metric $h\in \mathscr{G}_{\Omega}$ is 
a critical point of the variational problem \eqref{hilbert}   if and only if  there is a harmonic $2$-form $F$ with self-dual part 
$F^+\in \Omega$
such that the pair $(h,F)$ solves (\ref{closed}--\ref{energy}).
\end{prop}

One corollary is that an Einstein-Maxwell metric $h$ on a $4$-manifold must have constant scalar curvature;
indeed, if $h$ belongs to $\mathscr{G}_\Omega$, so does its entire conformal class, and
 the restriction of $\mathfrak{S}$ to a conformal class is exactly the variational problem used by 
 Yamabe to characterize metrics of constant scalar curvature \cite{bes}. At the other end of things, the 
 critical points of $\mathfrak{S}$ on the space of {\em all} Riemannian metrics are exactly the Einstein metrics, 
 and this provides further explanation for the fact that Einstein metrics  provide special  solutions of the   Einstein-Maxwell equations. 
 Yet another interesting consequence is the following:

 \begin{prop} \label{comps}
 Let $M$ and $\Omega$ be as above, and let 
   $$\mathscr{M}_\Omega =\{ \mbox{solutions $(h,F)$ of (\ref{closed}--\ref{energy}) on $M$}~|~ F^+\in \Omega\}/[ \Diff(M)\times \RR^+]$$
   be the 
 moduli space 
of $\Omega$-compatible solutions of the Einstein-Maxwell equations. Here $\Diff (M)$ denotes the 
group of diffeomorphisms of $M$ which act trivially on $H^2(M, \RR )$, and $\RR^+$ acts by rescaling $h$, without changing $F$.
 If $(h,F)$ and $(\tilde{h}, \tilde{F})$
 are solutions such that $s_hV_h^{1/2}\neq s_{\tilde{h}}V^{1/2}_{\tilde{h}}$, then these solutions belong to different connected components 
 of $\mathscr{M}_\Omega$. 
 \end{prop}
 \begin{proof} If $(h,F)\in \mathscr{G}\times \Gamma (\Lambda^2)$ is any smooth solution of (\ref{closed}--\ref{energy}), there is a finite-dimensional 
smooth submanifold $\mathscr{Y}\subset \mathscr{G}\times \Gamma (\Lambda^2)$ such that any other solution in a neighborhood
 $\mathscr{V}$ of $(h,F)$ is the pull-back of an element of $\mathscr{Y}$ via some diffeomorphism. To see this, let $\mathscr{S}$ be the set of 
 metrics of the form $\hat{h}=h+ \dot{h}$, where $\dot{h}$ is a symmetric and transverse  traceless  with respect to $h$, and  consider the smooth map 
 \begin{eqnarray*}
\mathscr{S}\times \Gamma (\Lambda^2) &\stackrel{\mathscr{E}}{\longrightarrow}& \Gamma (\odot^2_0\Lambda^1) \times \Gamma (\Lambda^2)\\
(\hat{h}, \hat{F} )&\longmapsto& \left(\mathring{r}_{\hat{h}}+ [\hat{F} \circ \hat{F} ]_{0,\hat{h}},  (d+\delta_{\hat{h}})^2 F \right)
\end{eqnarray*}
where $\hat{F}=F+\dot{F}\in \Gamma (\Lambda^2)$. 
Because of our assumption that $\dot{h}$ is transverse-traceless, we can modify the first term by adding $(\delta^*\delta_h)_{0}$ without 
altering the answer, but with the effect \cite{beeb} that the linearization of the equation at $h$ becomes elliptic. Let $\mathscr{T}$ denote image of 
the linearization $\mathscr{E}_{*(h,F)}$, and let $\wp: \Gamma (\odot^2_0\Lambda^1) \times \Gamma (\Lambda^2)
\to \mathscr{T}$ be the $L^2$-orthogonal projection. Then $\wp \circ \mathscr{E}$ is a submersion on some neighborhood $(h,F)$, and $\mathscr{Y} = (\wp \circ \mathscr{E})^{-1}(0)$
is a finite-dimensional manifold that contains a gauge-fixed version of any solution of (\ref{closed}--\ref{energy}) in  a neighborhood $\mathscr{V}$ of $(h,F)$. 
 
 Now choose a basis $\mathbf{a}_1, \ldots , \mathbf{a}_{b_2}$ for $H^2 (M)$, and, for each subset $\mathbf{I}\subset \{ 1, \ldots, b_2(M)\}$, let
 $\pi_\mathbf{I}: H^2(M)\to \RR^{|\mathbf{I}|}$ denote the map that sends the deRham class $\mathbf{a}$ to the relevant $|\mathbf{I}|$ components of $\mathbf{a}-\Omega$  relative to the chosen basis. 
 Let  $\Pi : \mathscr{Y}\to H^2 (M)$ be the smooth map defined by $(\hat{h}, \hat{F}) \mapsto [\hat{F}^+_{\hat{h}}]$,  let $\mathscr{Y}_\mathbf{I}\subset \mathscr{Y}$
 be the open subset of $\mathscr{Y}$ on which $\pi_\mathbf{I}\circ \Pi$ has derivative of rank $|\mathbf{I}|$, and  let $\mathscr{Z}_\mathbf{I}\subset \mathscr{Y}_\mathbf{I}$ be the smooth submanifold 
  defined by 
 $(\pi_\mathbf{I}\circ \Pi)^{-1}(0)$. 
 If $(\hat{h}, \hat{F})\in \mathscr{Y}$ then happens both to be 
 a genuine solution of (\ref{closed}--\ref{energy}) and to have the property that $\hat{h}\in \mathscr{G}_\Omega$, then at least one of the smooth manifolds 
 $\mathscr{Z}_\mathbf{I}$  will then have the property that its image in 
$\mathscr{G}$ passes through $\hat{h}$, with   $T_{\hat{h}} \mathscr{Z}_\mathbf{I}\subset T_{\hat{h}}\mathscr{G}_\Omega$; the value of $|\mathbf{I}|$ for which this occurs is just the 
rank of the derivative of $\Pi$ at $(\hat{h},\hat{F})$. 
It follows  that, for  any Einstein-Maxwell metric $\tilde{h}\in \mathscr{V}\cap \mathscr{G}_\Omega$, the real number $s_{\tilde{h}}V^{1/2}_{\tilde{h}}= \mathfrak{S}(\tilde{h})$
is a critical value of the pull-back  of $\mathfrak{S}$ to one of the  manifolds $\mathscr{Z}_\mathbf{I} \subset \mathscr{G}\times \Gamma (\Lambda^2)$.

Now since $\mathscr{G}\times \Gamma (\Lambda^2)$ is second countable, there is a countable collection $\mathscr{V}_j$ of such open sets, and an associated countable collection of 
finite-dimensional manifolds $\mathscr{Z}_{j,\mathbf{I}}$, such that every Einstein-Maxwell metric   $\tilde{h}\in \mathscr{G}_\Omega$
is the pull-back via some diffeomorphism of an Einstein-Maxwell metric $\hat{h}\in \mathscr{G}_\Omega$ through which passes some 
$\mathscr{Z}_{j,\mathbf{I}}$ with $T_{\hat{h}} \mathscr{Z}_{j,\mathbf{I}}\subset T_{\hat{h}}\mathscr{G}_\Omega$. This expresses the set of real numbers  occurring  as 
 $sV^{1/2}$ for  $\Omega$-compatible Einstein-Maxwell metrics as a subset of a countable union of  the critical values of specific smooth maps $\mathscr{Z}_{j,\mathbf{I}}\to \RR$, 
 where each $\mathscr{Z}_{j,\mathbf{I}}$ is a finite-dimensional smooth manifold. But Sard's theorem tells us that, for each $(j,\mathbf{I})$, the  set of  these 
 critical values has  measure
 zero in $\RR$, and the countable union of these over all $(j,\mathbf{I})$ therefore has measure zero, too. In particular, 
 the values of $sV^{1/2}$ occurring for $\Omega$-compatible Einstein-Maxwell metrics cannot contain an open interval. 
 Thus, if $t_1$ and $t_2$ are distinct values of $sV^{1/2}$ that occur for two different $\Omega$-compatible
Einstein-Maxwell metrics, there is a number $t_0$ between $t_1$ and $t_2$  which is {\em not} a value of $sV^{1/2}$ for any $\Omega$-compatible
Einstein-Maxwell metric. The open sets $(\mathfrak{S}|_{\mathscr{G}_\Omega})^{-1} (-\infty, t_0)$ and $(\mathfrak{S}|_{\mathscr{G}_\Omega})^{-1}  
(t_0,\infty)$ thus provide a separation of 
$\mathscr{M}_\Omega$ into two open disjoint sets, each of which contains just one of the given metrics. This shows that the two given metrics must belong to different connected components, as claimed. 
 \end{proof}

 The above argument avoids having to  show that the moduli space $\mathscr{M}_\Omega$
 is  locally smoothly path-wise connected. However, the latter does appear to be true, and to
  follow from a modification of the   argument  given by Koiso \cite{koimod} in the Einstein case.  
The idea is  to show  that  the  manifold $\mathscr{Y}$ appearing in the proof of  Proposition \ref{comps} carries a natural real-analytic structure, and that the Einstein-Maxwell equations then cut out  a real-analytic subset
of  $\mathscr{Y}$, thereby providing a real-analytic Kuranishi-type model for the local 
  pre-moduli space; cf.  \cite[Corollary 12.50]{bes}.  
 Ultimately, this is  related to the   fact
  that solutions of the Einstein-Maxwell equations are real-analytic in harmonic coordinates.
  The proof of the latter, which we leave to the interested reader,  merely consists of combining    \cite[Theorem 6.7.6]{morrey}
  with  \cite[Proposition 4]{lebem14}.

The physical motivation for the Einstein-Maxwell equations (\ref{closed}--\ref{energy}) 
 assigns a central role to the $2$-form $F$,  which represents an electromagnetic field. 
 It may therefore seem strange that we will often simply refer to 
 $h$ as an {\em Einstein-Maxwell metric}, without explicitly mentioning $F$. There is good reason for this terminology, however. Indeed, 
in  our Riemannian setting,   the metric $h$ essentially determines the relevant $2$-form:

 \begin{prop} 
 \label{lucid} 
 Suppose that $(M,h)$ is an Einstein-Maxwell manifold, where the $4$-manifold is connected and oriented. If $h$ is not actually Einstein, 
then the $2$-form $F$ needed to make $(h,F)$ solve the Einstein-Maxwell equations (\ref{closed}--\ref{energy})
is completely determined by $h$, modulo substitutions of the type  
$$F^+\rightsquigarrow {\mathbf c} F^+, \quad F^-\rightsquigarrow {\mathbf c}^{-1} F^-$$
where ${\mathbf c}\neq 0$ is a  real constant. 
\end{prop}
\begin{proof} First notice that \eqref{energy} can be rewritten as
 \begin{equation}
\label{matter}
\mathring{r}= -2 F^+\circ F^-,
\end{equation}
where $\mathring{r} = r-(s/4)h$ is the trace-free part of the Ricci tensor, and where $F^\pm = (F\pm \star F)/2$
are the self-dual and anti-self-dual  parts of $F$. If we assume that $h$ is not Einstein, then there exists an open ball  ${\mathscr B}$ on which 
$\mathring{r}\neq 0$,
and \eqref{matter} then algebraically  determines $F$ on  ${\mathscr B}$ up to substitutions of the form 
 $$F^+\rightsquigarrow u F^+, \quad F^-\rightsquigarrow u^{-1} F^-$$
for a smooth non-vanishing function $u$ defined on ${\mathscr B}$. However, equations (\ref{closed}-\ref{coclosed}) imply that $F^+$ is closed, and requiring that 
$\tilde{F}^+:=uF^+$  also be closed then results in the condition that 
$$0= d(u F^+) = du \wedge F^+ + u ~dF^+ = du \wedge F^+.$$
However,  a self-dual form is non-degenerate on the set where it is non-zero, and $F^+$ is non-zero on ${\mathscr B}$ by hypothesis. It therefore follows 
 that $du=0$ on  ${\mathscr B}$, and hence that $u={\mathbf c}$ on this open ball. 
Thus, if  $h$ is not actually an Einstein metric, any other candidate $\tilde{F}^+$ 
for $F^+$ would have to coincide with ${\mathbf c}F^+$, for some constant  ${\mathbf c}$, on a non-empty open set  ${\mathscr B}\subset M$. But since 
$\tilde{F}^+- {\mathbf c}F^+$ then belongs to the kernel of  $d+ d^*$ and vanishes on an open set of our connected $4$-manifold $M$, unique continuation for
harmonic forms \cite{arons} then implies that $\tilde{F}^+\equiv {\mathbf c} F^+$ on all of $M$. Applying the same argument to  $F^-$,  
we thus see that if the Einstein-Maxwell metric $h$ is not actually Einstein, then the metric $h$ determines the closed and co-closed $2$-form 
$F=F^++F^-$ modulo changes of the type 
 $$F^+\rightsquigarrow {\mathbf c} F^+, \quad F^-\rightsquigarrow {\mathbf c}^{-1} F^-$$
 for  a non-zero constant ${\mathbf c}$.
 \end{proof} 
  Of course, the Einstein case is  exceptional. Indeed,  one obtains a solution of the Einstein-Maxwell equations 
 (\ref{closed}--\ref{energy}) on any oriented Einstein $4$-manifold by just taking $F$ to be an arbitrary   self-dual (or anti-self-dual) harmonic 
  $2$-form. 
  
  This investigation will actually focus on a special class of solutions of the Einstein-Maxwell equations, first introduced in \cite{lebem14}:
  
  \begin{defn} Let $(M^4,J)$ be a complex surface. A solution $(h,F)$ of the  Einstein-Maxwell equations (\ref{closed}--\ref{energy}) 
  on $(M,J)$ will be called
  {\em strongly Hermitian} if both 
 $h$ and $F$ are  invariant under the action of the integrable almost-complex structure $J$:
\begin{eqnarray*}
h&=& h(J\cdot , J\cdot ) , \\
F&=& F ( J\cdot , J \cdot ).
\end{eqnarray*}
 When this happens, we will then say that $h$ is a {strongly Hermitian Einstein-Maxwell metric} on $(M,J)$. 
    \end{defn}

The following is one   of the principal results of \cite{lebem14}:

 \begin{prop} Let $(M^4,J)$ be a compact complex surface. Then $h$ is a strongly Hermitian Einstein-Maxwell metric on $(M,J)$ iff there 
 is a K\"ahler metric $g$ on $(M,J)$ and a real holomorphy potential $f\neq 0$ on $(M,g,J)$ such that $h=f^{-2}g$ has constant scalar curvature. 
 \end{prop}
 
 Here a real holomorphy potential $f$ is a real-valued function  whose gradient with respect to the K\"ahler metric $g$ is the real part of a holomorphic vector field;
 this is equivalent to requiring either that $J\grad f$ be a Killing field of $g$, or that the Riemannian Hessian $\nabla df$  of $f$ be $J$-invariant. 
 The harmonic $2$-form $F$ needed to solve  (\ref{closed}--\ref{energy})  in conjunction with $h=f^{-2}g$ can then be taken to be 
$$F  = \frac{\omega}{2} +  \left(\rho + 2if^{-1} \partial \bar{\partial} f\right)^-$$
where $\omega= g(J\cdot , \cdot )$ and $\rho = r_g(J\cdot, \cdot)$ are the K\"ahler and   Ricci forms  of $g$, respectively, 
 and where the final superscript indicates projection to the anti-self-dual part of a
$2$-form. 
However, 
this choice of $F$  is   of course not quite unique, and can  be modified   in the manner described by Proposition \ref{lucid}. 
 
 \section{Solutions  on Spherical Shells} 

In this section, we describe an essentially local construction of  Einstein-Maxwell metrics on a spherical shell $S^3\times I$, 
where $I\subset \RR$ is 
an open interval. We assume from the outset that the metrics in question are {\em cohomogeneity one} \cite{grozi}, with an isometric action of
${\mathbf U}(2)$ whose generic orbit is the $3$-sphere $S^3={\mathbf U}(2)/{\mathbf U}(1)$. Such a metric induces
a homogeneous metric on $S^2$, which must be some number $\varrho^2/4$ times the standard unit-sphere metric. At least generically, we
can then use the positive function $\varrho$ as a coordinate, and so write our metric in so-called {\em Bianchi IX form}

$$g= \frac{d\varrho^2}{\Phi (\varrho)} + \varrho^2 \left[ \sigma_1^2 + \sigma_2^2 + \tilde{\Phi} (\varrho) \sigma_3^2\right]$$
where $\sigma_1, \sigma_2, \sigma_3$ is a left-invariant co-frame on $S^3$ which is orthonormal with respect to the usual metric
on $S^3= \mathbf{SU}(2)= \mathbf{Sp} (1)$. The fact that the metric coefficients of $\sigma_1$ and $\sigma_2$ are equal reflects the fact that   
 ${\mathbf U}(2)$ acts on $S^3$ with isotropy  subgroup ${\mathbf U}(1)$, and the latter acts on a cotangent space  by rotations in $\sigma_1$ and $\sigma_2$. 
On the other hand,   $\Phi$ and $\tilde{\Phi}$ are  for the moment completely arbitrary positive functions.

The structure equations for $\mathbf{SU}(2)$  tell us that our co-frame satisfies 
$$d\sigma_1 = 2\sigma_2\wedge \sigma_3,  \quad d\sigma_2 = 2\sigma_3\wedge \sigma_1, \quad d\sigma_3 = 2\sigma_1\wedge \sigma_2.$$
Notice, in particular, that $\sigma_1\wedge \sigma_2$ is closed. In fact, this closed $2$-form is simply  the pull-back of the area $2$-form of the 
curvature $4$ metric on $S^2$. 

This picture automatically provides us with a $g$-compatible almost-complex structure $J$, characterized by 
$$ \frac{d\varrho}{\sqrt{\Phi (\varrho)\tilde{\Phi} (\varrho)}} \mapsto \sigma_3, \quad \sigma_1 \mapsto \sigma_2.$$
This actually  turns our spherical shell into a complex manifold. 

\begin{lem} The almost-complex structure $J$ is integrable, and so makes $g$ into a Hermitian metric. 
\end{lem}
\begin{proof}
With respect to the given $J$,  the  $(1,0)$-forms are spanned by 
$$\frac{d\varrho}{\sqrt{\Phi (\varrho)\tilde{\Phi} (\varrho)}} + i \sigma_3 \quad \mbox{and} \quad  \sigma_1 + i \sigma_2.$$
Let $\mathscr{I}$ denote the differential ideal generated by these two $1$-forms. 
Because 
$$d \left( \frac{d\varrho}{\sqrt{\Phi (\varrho)\tilde{\Phi} (\varrho)}} + i \sigma_3 \right) = 2i \sigma_1 \wedge \sigma_2 =  (\sigma_1 + i\sigma_2) \wedge (2i \sigma_2)$$
and 
$$d \left(  \sigma_1 + i \sigma_2\right) = 2 \sigma_2 \wedge \sigma_3 + 2i \sigma_3 \wedge \sigma_1 = (\sigma_1 + i\sigma_2)  \wedge (-2i \sigma_3),$$
we have $d  \mathscr{I}\subset \mathscr{I}$. Thus $\mathscr{I}$ is a closed differential ideal, and 
$[T^{0,1},T^{0,1}]\subset T^{0,1}$. Hence $J$ is integrable, in the sense of the 
 Newlander-Nirenberg theorem \cite{newnir}. 
\end{proof}

Imposing the K\"ahler condition now gives us a simple constraint. 

\begin{lem} The metric $g$ is K\"ahler with respect to $J$ if and only if  $$\Phi \equiv  \tilde{\Phi}.$$ 
\end{lem}
\begin{proof}
The associated  $2$-form of $(g,J)$ is 
$$\omega = \sqrt{  \frac{\tilde{\Phi}}{\Phi}} ~\varrho ~d\varrho \wedge \sigma_3 + \varrho^2 \sigma_1 \wedge \sigma_2.$$
Thus 
\begin{eqnarray*}
d \omega &=& - \sqrt{  \frac{\tilde{\Phi}}{\Phi}   } \varrho ~d\varrho \wedge d\sigma_3 + 2\varrho d\varrho \wedge \sigma_1 \wedge \sigma_2 + \varrho^2 d(\sigma_1 \wedge \sigma_2)\\
&=& \left( 1- \sqrt{  \frac{\tilde{\Phi}}{\Phi }   }\right) 2\varrho ~d\varrho \wedge \sigma_1 \wedge \sigma_2 ,  
\end{eqnarray*}
showing that   $(g,J)$ is K\"ahler iff $\tilde{\Phi}\equiv \Phi$. 
\end{proof}

We will  henceforth assume that $g$ is K\"ahler,  and so  given by 
\begin{equation}
\label{ansatz}
g= \frac{d\varrho^2}{\Phi (\varrho)} + \varrho^2 \left[\sigma_1^2 + \sigma_2^2 + {\Phi} (\varrho) \sigma_3^2\right]
\end{equation}
with  K\"ahler form
$$\omega = \varrho ~d\varrho\wedge \sigma_3 + \varrho^2 \sigma_1 \wedge \sigma_2.$$
In this context,  the Killing field $\xi$ defined by 
$$ \sigma_3 (\xi ) = 1, \quad  d\varrho  (\xi ) = \sigma_1(\xi )  = \sigma_2 (\xi )  =0$$
acquires considerable interest, as it 
 preserves $\omega$, and is generated by  the Hamiltonian 
$$t= \frac{\varrho^2}{2}.$$
This observation immediately  gives $\varrho$ a more  global and intrinsic meaning than 
our previous provisional definition might indicate. In particular, the   fact that  the 
symplectic reductions have area  $2\pi t$ now becomes a special case of the Duistermaat-Heckman formula \cite{duiheck}.

Treating the Hamiltonian $t$ as a coordinate  now allows us to    put our K\"ahler  metric $g$ into  the standard form \cite{mcp2} 
$$g = w~\check{g} + w ~dt^2 + w^{-1} \theta^2,$$
where $\check{g}$ is $w^{-1}$ times the usual metric on the symplectic reduction, and 
where $\theta (\xi ) = 1$. Comparing this with \eqref{ansatz}, we now immediately
 see that $\theta = \sigma_3$ and that  $w^{-1} = \varrho^2 \Phi$.
It then follows   that 
$$\check{g} = w^{-1}\left[  \varrho^2 (\sigma_1^2 + \sigma^2 )\right] = \varrho^4 \Phi (\sigma_1^2 + \sigma^2 )$$
is  a metric of Gauss curvature 
$$\check{K} = \frac{4}{\varrho^4 \Phi} = \frac{1}{t^2 \Phi}$$
and K\"ahler form 
$$\check{\omega} = \varrho^2 \Phi \sigma_1 \wedge \sigma_2= 4t^2 \Phi \sigma_1 \wedge \sigma_2$$
on the $2$-sphere $S^2$. 
The formalism of \cite{mcp2} therefore allows us to  calculate the scalar curvature $s$ or $g$, using the general formula \cite{mcp2,lebtoric,ls} 
$$s~d\mu = \left[ 2 \check{K} \check{\omega} - \frac{d^2}{dt^2} \check{\omega}\right] \wedge dt \wedge \theta,$$
where $d\mu = \omega^2/2$ is the volume form of $g$. In our case, the latter is explicitly given by 
$$d\mu = \varrho^3 ~d\varrho\wedge \sigma_1 \wedge \sigma_2 \wedge \sigma_3 = 2t ~dt\wedge \sigma_1 \wedge \sigma_2 \wedge \sigma_3,$$
so we have
$$ 2t~s ~dt\wedge \sigma_1 \wedge \sigma_2 \wedge \sigma_3= \left[  \frac{2}{t^2 \Phi} 4t^2 \Phi  - \frac{d^2}{dt^2} (4 t^2 \Phi )\right] \sigma_1 \wedge\sigma_2 \wedge dt \wedge \sigma_3,$$
and hence 
\begin{equation}
\label{key}
s = \frac{4}{t}  - \frac{2}{t}  \frac{d^2}{dt^2} ( t^2 \Phi )= \frac{2}{t}\frac{d^2}{dt^2} \left[ t^2(1 -   \Phi )\right].
\end{equation}

As an simple  illustration of this  formula, let us now use this to determine when $g$ is an {\em extremal} K\"ahler metric, in the sense of Calabi \cite{calabix}.

\begin{lem} \label{xxx}
The K\"ahler metric $(g,J)$ defined by \eqref{ansatz}  is extremal iff
$$s = \mathfrak{a}\, t + \mathfrak{b}$$
for real constants $\mathfrak{a}$ and $\mathfrak{b}$, where $t=\varrho^2/2$.
\end{lem} 
\begin{proof} Let $\eta = J \nabla s$, where $s$ is the scalar curvature of $g$. Then $g$ is extremal 
iff $\eta$ is a Killing field.  However,  $s$ is invariant under the isometry group, and is therefore a function of $\varrho$, or equivalently, a function of $t$. 
Thus, we automatically have $\eta = u \xi$, where $u = ds/dt$, and where 
 $\xi  = J\nabla t$ is already known to be a Killing field. If $u$ is constant, then 
$\eta$ is  a Killing field, and $s$ is an affine-linear function of $t$. Conversely, if $\eta$ is a Killing field, then 
$$0 = \nabla^{(a}\eta^{b)}= \nabla^{(a}u\xi^{b)} = u  \nabla^{(a}\xi^{b)} + \xi^{(b} \nabla^{a)}u =  \xi^{(a} \nabla^{b)}u,$$
and it follows that $\nabla u$ must vanish, since the symmetric tensor product of two non-zero vectors is always non-zero. This shows that   $u$ must be constant, 
and  that $s$ must therefore be an affine-linear function of $t$, as claimed. 
\end{proof}

\begin{prop}  \label {extremal}
The K\"ahler metric $(g,J)$ defined by \eqref{ansatz}  is extremal iff
\begin{eqnarray*}
\Phi &=&  At^2 + Bt +1+ Ct^{-1} + Dt^{-2} \\
       &=&  \frac{A}{4}\varrho^4 + \frac{B}{2}\varrho^2+1+ \frac{2C}{\varrho^2} + \frac{4D}{\varrho^4}  
\end{eqnarray*}
for  real constants $A$, $B$, $C$, and $D$, subject only to the constraint that $\Phi> 0$ in the region of interest. 
Moreover, the scalar curvature of $g$ is then given by $$s=-12(2At+B)= -12 (A\varrho^2 + B).$$
\end{prop}
\begin{proof} By Lemma \ref{xxx} and equation \eqref{key}, 
the extremal condition is equivalent to 
$$\frac{d^2}{dt^2} \left [ t^2 (\Phi -1 ) \right] = -\frac{t}{2} (\mathfrak{a}\, t + \mathfrak{b})$$
and integrating twice therefore yields
$$t^2 (\Phi -1 ) = -\frac{\mathfrak{a}}{24}t^4 - \frac{\mathfrak{b}}{12}t^3+ Ct+D.$$
Setting $\mathfrak{a}= -24A$ and $\mathfrak{b}= -12B$ then yields the result. 
\end{proof} 

In principle, these metrics must coincide with the metrics found by Calabi \cite{calabix} in complex coordinates; however, their much simpler appearance 
 in the current formalism will turn out to be very useful for our purposes.   Notice that 
$g$ has constant scalar curvature iff $A=0$, and is scalar-flat K\"ahler iff $A=B=0$. The latter metrics were first introduced in \cite{lpa}.
The metric is Ricci-flat if $A=B=C=0$, in which case $g$  becomes  the Eguchi-Hanson metric \cite{EH},  unless $A=B=C=D=0$,  when it is flat.

We now turn to the problem of constructing strongly Hermitian solutions of the Einstein-Maxwell equations. By  \cite[Theorem A]{lebem14},
such metrics are exactly those of the form  $h=f^{-2}g$,  where  $g$ is a  K\"ahler metric , $f> 0$ is a real   holomorphy potential, and 
$h=f^{-2}g$ has constant scalar curvature, except in the exceptional case that $h$ is anti-self-dual and Einstein. Here a real holomorphy potential means
  a real-valued function $f$ such that $J\nabla f$ is a Killing field. 

Our eventual goal is to construct strongly Hermitian Einstein-Maxwell  metrics which are ${\mathbf U}(2)$-invariant and live on a compact complex surface $(M^4,J)$. 
In this setting, the K\"ahler form $\omega$ of $g$ will be harmonic with respect to $h$, and will necessarily be the unique harmonic form its
de Rham cohomology. This means that $\omega$ will necessarily be ${\mathbf U}(2)$-invariant, and hence that $f=2^{-1/4}|\omega |_h^{1/2}$
will necessarily be ${\mathbf U}(2)$-invariant, too. The K\"ahler metric $g=f^2h$ is therefore ${\mathbf U}(2)$-invariant, too, so we 
can locally represent our metric in the form \eqref{ansatz}.
Since  $f$ is a function on the space of ${\mathbf U}(2)$-orbits, it must, in 
our local picture,  be a function of $t$. However, both $t$ and $f$ are then Hamiltonians whose symplectic gradients are Killing fields, 
and the same argument used to prove Lemma \ref{xxx} therefore implies that $f$ must be an affine function of $t$. Since we are only 
interested in solutions which are not cscK, this means that $f$ must take the form ${\mathfrak c}t+ {\mathfrak d}$, where ${\mathfrak c}\neq 0$.
However, multiplying $f$ by a non-zero constant just rescales $h$ into another solution of the same problem, so we can henceforth take $f$ to be
of the form
$$f= t - \alpha$$
for some real constant $\alpha$.

Requiring  that $h=f^{-2}g$ have constant scalar curvature then amounts to saying that 
$$(6\Delta + s) f^{-1} = \kappa f^{-3},$$
or equivalently as 
\begin{equation}
\label{scale}
s =  \kappa f^{-2} - 6 f \Delta f^{-1},
\end{equation}
where $\Delta = -\nabla \cdot \nabla= - \star d \star d $ is the positive Laplacian, and $\kappa$ is some real constant. 
Rewriting \eqref{ansatz} as
$$g = \frac{dt^2}{2t\Phi } + 2t \left[\sigma_1^2 + \sigma_2^2 + {\Phi}  \sigma_3^2\right]$$
with 
$$\omega = dt \wedge \sigma_3 + 2t \sigma_1 \wedge \sigma_2,$$
the fact that 
$$\frac{dt}{\sqrt{2t\Phi}} , \sqrt{2t\Phi}\sigma_3 , \sqrt{2t} \sigma_1, \sqrt{2t} \sigma_2$$
is an oriented orthonormal basis tells us that 
$$\star dt = 4t^2 \Phi \sigma_1 \wedge \sigma_2 \wedge \sigma_3, $$
so  that, for any function $\varphi (t)$,  
$$\star d \varphi(t) = 4  t^2 \Phi \varphi^\prime (t) \sigma_1  \wedge \sigma_2 \wedge \sigma_3 .$$
Thus 
$$d \star d \varphi = 4 [t^2 \Phi \sigma_1 \varphi^\prime]^\prime dt \wedge \sigma_1  \wedge \sigma_2 \wedge \sigma_3, $$
and
$$\Delta \varphi = - \star d \star d \varphi = - \frac{2}{t} [t^2 \Phi  \varphi^\prime]^\prime$$
for any function of $t$, where primes denote derivatives with respect to $t$.
Setting 
$$\Psi := t^2 \Phi ,$$
and setting $f= t-\alpha$, we can thus rewrite \eqref{scale} as 
$$s = \frac{\kappa}{(t-\alpha)^2} - \frac{12}{t} (t-\alpha)  \left[ \frac{\Psi}{(t-\alpha)^2} \right]^\prime. $$
However, 
\eqref{key} tells us that 
$$s= \frac{2}{t} \left[ 2 - \Psi^{\prime\prime}\right] ,$$
and equating these two expressions thus tells us that 
$$2 - \Psi^{\prime\prime} = \frac{\kappa t}{2 (t-\alpha)^2} - 6 (t-\alpha) \left[ \frac{\Psi}{(t-\alpha)^2} \right]^\prime.$$
In other words, we will obtain a conformally K\"ahler solution of the Einstein-Maxwell equations iff $\Psi$ solves the 
linear inhomogeneous equation 
\begin{equation}
\label{linear}
(t-\alpha)^2 \Psi^{\prime\prime}  - 6 (t-\alpha ) \Psi^\prime + 12 \Psi = 2 (t-\alpha)^2 - \frac{\kappa}{2} (t-\alpha) - \frac{\kappa\alpha }{2}.
\end{equation}
However, the linear operator 
$$
\Psi \longmapsto (t-\alpha)^2 \Psi^{\prime\prime}  - 6 (t-\alpha ) \Psi^\prime + 12 \Psi 
$$
acts on powers of  $(t-\alpha)$ by 
$$
(t-\alpha)^\ell \longmapsto  [\ell (\ell -1) - 6 \ell + 12] (t-\alpha)^\ell = (\ell - 4) (\ell -3) (t-\alpha)^\ell  ,
$$
so the general solution of \eqref{linear} is 
\begin{eqnarray}
\Psi &=& \mathfrak{A} (t-\alpha)^4 + \mathfrak{B} (t-\alpha)^3 + (t-\alpha)^2 - \frac{\kappa}{12} (t-\alpha) - \frac{\kappa\alpha }{24} \label{solved}\\
 &=& \mathfrak{A} x^4 + \mathfrak{B} x^3 + x^2 + {\mathfrak C} x + \frac{{\mathfrak C}\alpha}{2} \label{redux} 
\end{eqnarray}
where,  for clarity,  we have set $x=t-\alpha$ and ${\mathfrak C}= -\kappa/12$. 
Notice that \eqref{solved}  is not quite the general quartic function of $t$, because  the five coefficients only depend  on the four constants $\mathfrak{A}$, $\mathfrak{B}$, 
$\alpha$, and $\kappa$. Also notice that the polynomial  \eqref{redux} in $x$  completely determines $\alpha$  when ${\mathfrak C}\neq 0$, thereby allowing one to reconstruct 
\eqref{solved}  from a generic quartic polynomial in $x$ for which the coefficient of $x^2$ is $1$. 

Making systematic use of the  new variable $x=f=t-\alpha$ now allows us to put the above results  in a simple, concise form. 
In these terms, the general solution of our problem is provided by the K\"ahler metric 
\begin{equation}
\label{explicit}
g= (x+\alpha ) \left[\frac{dx^2}{2\Psi} + 2(\sigma_1^2+\sigma_2^2)\right] + \frac{2\Psi}{x+\alpha} \sigma_3^2
\end{equation}
associated to a quartic polynomial $\Psi$ of the special form \eqref{redux}. The associated Einstein-Maxwell metric is then given by 
\begin{equation}
\label{induced}
h= \frac{g}{x^2}.
\end{equation}

\begin{prop} \label{criterion}
 Let  $\Psi$ be a quartic polynomial in $x$ of the special form  \eqref{redux}, with  ${\mathfrak A}$ and $\alpha$ both non-zero. 
Let $g$ and $h$ be the corresponding  K\"ahler and  Einstein-Maxwell metrics defined by \eqref{explicit} and \eqref{induced} on a spherical shell where $x$, $x+\alpha$, and 
$\Psi$ are all positive. 
 Then the following are equivalent:
\begin{enumerate}[(i)]
\item The Hermitian metric $h$ is Einstein. \label{eins}
\item The conformal  class $[g]=[h]$ is Bach-flat. \label{zwei} 
\item The K\"ahler metric $g$ is extremal. \label{drei}
\item ${\mathfrak B}= 2{\mathfrak A}\alpha$. \label{vier} 
\end{enumerate}
\end{prop}
\begin{proof}
 By  Proposition \ref{extremal}, with $\Psi= t^2\Phi$, the metric $g$ is extremal iff 
 $$\Psi= {\mathfrak A}x^4 + {\mathfrak B}x^3 + x^2 +\cdots = At^4 + B t^3 + t^2 + \cdots$$
After making the substitution $x= t-\alpha$ and then comparing the coefficients of $t^2$, we obtain 
 $${\mathfrak A}(6\alpha^2) + {\mathfrak B} (-3\alpha)=0, $$
 and, since we have assumed that  $ {\mathfrak A}$ and $\alpha$ are non-zero, this happens iff 
  \begin{equation}
\label{skol}
\alpha = \frac{{\mathfrak B}}{2{\mathfrak A}},
\end{equation}
thus showing  that  (\ref{drei}) $\Longleftrightarrow$ (\ref{vier}). 
Similarly, by comparing the coefficients of $t^4$ and $t^3$,  we also have $A= {\mathfrak A}$ and 
$B={\mathfrak B}- 4A\alpha$.  
In the extremal case, equation  \eqref{skol} and  Proposition \ref{extremal}  therefore  tell us that 
$$s=  12 (At+B) = 24 A\left(t + \frac{B}{2A}\right) = 24 A\left(t+ \frac{{\mathfrak B}}{2{\mathfrak A}}- 2\alpha\right) = 24A( t-\alpha ) ,$$  
which is to say  that $s$ is  a non-zero constant times $x=t-\alpha$. However, a result of Derdzi\'nski \cite[Proposition 4]{derd} asserts that if $g$ is an extremal 
K\"aher metric in real dimension $4$, with non-constant scalar curvature $s$, then $s^{-2}g$ is Einstein iff the latter metric has constant scalar curvature. 
But   $h=x^{-2}g$ has constant scalar curvature $s_h=\kappa$ by construction, so this shows that (\ref{drei}) $\Rightarrow$ (\ref{eins}). 
On the other hand, in real dimension $4$, any Einstein metric is Bach-flat, and any Bach-flat K\"ahler metric is extremal. Hence 
 (\ref{eins}) $\Rightarrow$  (\ref{zwei}) $\Rightarrow$ (\ref{drei}), and we are done. 
\end{proof}

\begin{prop} \label{switch} 
Let $h$ be any Einstein-Maxwell metric on a spherical shell arising by rescaling the K\"ahler metric associated with a quartic polynomial $\Psi (t)$. 
Then $h$ is also obtained by rescaling a second K\"ahler metric $\hat{g}$, which is instead compatible with an {\em oppositely oriented}  complex structure 
on the shell. Moreover, in inverted coordinates, $\hat{g}$ is associated with the quartic polynomial
$$\tilde{\Psi}(t) = t^4 \Psi (t^{-1}).$$
\end{prop}
\begin{proof} If a $\mathbf{U}(2)$-invariant  K\"ahler metric $g$ is expressed as 
$$
g= t \left[  \frac{dt^2}{2\Psi(t)} + 2 (\sigma_1^2 + \sigma_2^2)\right]   +  \frac{2\Psi (t)}{t} \sigma_3^2
$$
then the substitution  $\mathfrak{t}= 1/t$ yields
\begin{eqnarray*}
g&=& \frac{1}{\mathfrak{t}^2}\left( \mathfrak{t} \left[  \frac{d\mathfrak{t}^2}{2\mathfrak{t}^4\Psi (\mathfrak{t}^{-1})} + 2 (\sigma_1^2 + \sigma_2^2)\right]   +  \frac{2\mathfrak{t}^4\Psi (\mathfrak{t}^{-1})}{\mathfrak{t}} \sigma_3^2 \right)\\
 &=& \frac{1}{\mathfrak{t}^2}\left( \mathfrak{t} \left[  \frac{d\mathfrak{t}^2}{2\tilde{\Psi} (\mathfrak{t})} + 2 (\sigma_1^2 + \sigma_2^2)\right]   +  
 \frac{2\tilde{\Psi} (\mathfrak{t})}{\mathfrak{t}} \sigma_3^2 \right)
\end{eqnarray*}
where $\tilde{\Psi}(\mathfrak{t}) := \mathfrak{t}^4 \Psi (\mathfrak{t}^{-1})$. This shows 
that the metric $\hat{g}= \mathfrak{t}^2 g= g/t^2$ is also K\"ahler, although instead compatible with an oppositely oriented  complex structure.
Moreover, when $\Psi$ is a quartic polynomial, $\tilde{\Psi}(\mathfrak{t}) = \mathfrak{t}^4 \Psi (\mathfrak{t}^{-1})$ is once again a 
quartic polynomial. 
\end{proof}

We note in passing that if 
$h=(t-\alpha)^{-2}g$ for some $\alpha> 0$, then $h= (1-\alpha/t)^{-2} \hat{g} = \hat{\alpha}^{2} (\mathfrak{t}-\hat{\alpha})^{-2} \hat{g}$, where $\hat{\alpha}:=1/\alpha$. 
Requiring that $h$ have constant scalar curvature is  thus equivalent  either to stipulating  that $\Psi$ take the form \eqref{solved}, or  to  requiring that
 the expansion of $\tilde{\Psi}$ 
in $(\mathfrak{t}-\hat{\alpha})$ be analogously constrained. We leave it as an exercise for the interested reader to  verify by direct calculation  that these  
algebraic constraints on the quartics $\Psi$ and $\tilde{\Psi}$
are indeed  equivalent.

\section{Solutions on Compact $4$-Manifolds}
\label{compact}

In the previous section, we produced  a family of Einstein-Maxwell metrics on spherical shells $(a,b) \times S^3$, where 
for simplicity, we  now systematically use  $x=t-\alpha$ as the ``radial'' variable on our shell,
so that $g$  is given by \eqref{explicit}, wiith $x\in (a,b)$, and $h$ is given by \eqref{induced}. We will now next seek to ascertain
when  some  $\ZZ_k$-quotient of $(a,b) \times (S^3/\ZZ_k)$ of such a shell has  a  metric-space completion which is a  compact Riemannian manifold,
where  the $\ZZ_k$-action  is generated by $\exp (2\pi \xi/ k)$.

\begin{prop} \label{desiderata} 
Let $\Psi (x)$ be a quartic polynomial of the form \eqref{redux}. Suppose that 
 $\mathfrak{A}\neq 0$,  that $a$ and $b$ have the same sign,  that $x+\alpha > 0$ for $x\in [a,b]$, and that 
$\Psi (x) > 0$ for $x\in (a,b)$. If
$$\Psi (a) = \Psi (b) =0, \quad \Psi^\prime|_{x=a}= k (a+\alpha), \quad \mbox{and} \quad \Psi^\prime|_{x=b}= -k (b+\alpha),$$
then 
 the metric 
$g$ defined on $(a,b) \times (S^3/\ZZ_k)$ by \eqref{explicit} extends to a  K\"ahler metric on a compact complex manifold $(M,J)$ obtained by adding two copies of $\CP_1$,
one  at $x=a$ and one at 
$x=b$. This metric is invariant under an isometric action of $\mathbf{U}(2)$ on $(M,g)$, and \eqref{induced} defines a strongly Hermitian Einstein-Maxwell metric
$h$ on $(M,J)$. 
\end{prop}
\begin{proof}
Under the composition of the Hopf map $S^3/\ZZ_k\to \CP_1$ and the factor projection $(a,b) \times (S^3/\ZZ_k)\to (S^3/\ZZ_k)$, our metric $g$ lives 
on an annulus bundle over $\CP_1$. However, if we now choose  to view each fiber annulus as a twice-punctured $2$-sphere, our hypotheses then  
guarantee that fiber-wise metric extends smoothly to a smooth metric on this $S^2$. Indeed, the fiber-wise metric takes the form 
$$g|_{\mbox{\tiny fiber}}= \frac{dx^2}{\Upsilon(x)} + \Upsilon(x) d\vartheta^2 = \frac{dx^2}{\Upsilon(x)} +  \frac{\Upsilon(x)}{k^2} d\hat{\vartheta}^2$$
where $\Upsilon=  2\Psi/(x+\alpha)$, $d\vartheta = \sigma_3|_{\mbox{\tiny fiber}}$, and  $\hat{\vartheta} = k \vartheta$. 
Our hypotheses guarantee that $\Upsilon (x) = 2k (x-a) + O((x-a)^2)$ and that $\Upsilon^\prime  = 2k  +  O (x-a)= 2k + O(\Upsilon )$, where the error terms are   rational functions 
of $x-a$ which are 
regular at $x-a=0$, and so are real-analytic functions of  $\Upsilon$ in a  neighborhood of $0$. On an interval $x\in (a, a+\varepsilon )$ where
$\Upsilon$ is increasing, let us therefore choose  
$$\rad= \frac{\sqrt{\Upsilon}}{k}$$
as a new ``radial'' coordinate. We then have  $d\rad= [\Upsilon^\prime dt]/(2k\sqrt{\Upsilon})$, so the fiber metric becomes
$$
g|_{\mbox{\tiny fiber}}= \left(\frac{2k}{\Upsilon^\prime}\right)^2 d\rad^2 + \rad^2 d\hat{\vartheta}^2 = (1+ \Pi (\rad^2 )) d\rad^2 + \rad^2 d\hat{\vartheta}^2
$$
for some real-analytic  function $\Pi (\mathfrak{u})$ which vanishes at $\mathfrak{u}=0$. It follows that the fiber metric is a real-analytic Riemannian metric in a neighborhood of the puncture $x=a$. 
The same argument with $x-a$ replaced by $b-x$,  similarly shows that the fiber metric extends real-analytically across the puncture $x=b$.

Regularity of the remaining components of the metric $g$ is now straightforward. Indeed, notice that $x$ is a real-analytic function of $\rad^2$, 
where $\rad$ is the local radial fiber coordinate introduced above. Since $\sigma_1^2 + \sigma_2^2$ is the standard curvature $4$
metric on $\CP_1$, the terms in the metric gotten by multiplying  $\sigma_1^2 + \sigma_2^2$  by $2(x+\alpha)$ are therefore real-analytic. 
The rest of the metric is then just obtained by extending the fiber metric to $TM$ by taking it to annihilate the horizontal space
of the standard homogeneous connection on the Chern class $\pm k$ disk bundle over $\CP_1$, so the resulting metric $g$ is actually 
real-analytic on $M$. By the same argument, 
the K\"ahler form
$$\omega = dx \wedge \sigma_3 + 2 (x+ \alpha ) \sigma_1\wedge \sigma_2 ,$$
of $g$ also extends real-analytically to $M$; and since $\nabla \omega =0$ on an open set of $M$, it  follows that $\omega$
is a parallel form on $(M,g)$. Thus $g$ is actually a K\"ahler metric on $M$. Moreover, the fiber $2$-spheres are holomorphic curves on an
open dense set, and hence everywhere by continuity. Their Riemannian normal bundles are therefore $J$-invariant, and we therefore
see that the two copies of $\CP_1$ we have added at $x=a$ and $x=b$ are now actually holomorphic curves. The original isometric action
of $\mathbf{U}(2)$ on the shell also acts isometrically along these added curves, and so defines a global isometric action on $(M,g)$. 
Finally, the real-analytic function $x$ on $M$ is a holomorphy potential on a dense set, and hence everywhere, while   the globally defined  Hermitian metric
$h= x^{-2} g$ has constant scalar curvature on an open dense set, and hence everywhere. It follows that $h$ is a strongly Hermitian 
Einstein-Maxwell metric on $(M,J)$. 
\end{proof}

It remains for us to try to construct the most general quartic polynomial $\Psi (x)$ with the required properties. 
We begin by choosing two  real numbers $b> a$. We will  then try  to arrange for $x=a$ and $x=b$ to be two successive
zeros of $\Psi$ by setting 
$$\Psi  = (b-x) (x-a) Q(x)$$
for some quadratic polynomial $Q$ which is positive on $[a,b]$.  In order to ensure that these zeros merely correspond to coordinate singularities, though, 
Proposition \ref{desiderata} insists that 
we  stipulate  that 
$$\Psi^\prime|_{x=a}= k (a+\alpha), \quad \Psi^\prime|_{x=b}= -k (b+\alpha)$$
and this now becomes the requirement that 
$$Q(x) = \frac{k(x+\alpha)+E (b-x) (x-a)}{b-a}  $$
for some positive  integer $k$ and some real constant $E$. 
Since $h=g/x^2$, we must also, in keeping with Proposition \ref{desiderata},  require that $a$ and $b$  have the {\em same sign}, so
as to guarantee  that  $x^2$ will be positive on 
$[a,b]$. 
Our challenge is now to arrange   for the resulting 		quartic 
\begin{equation}
\label{constrained}
\Psi  = \frac{(b-x) (x-a)}{b-a}  \left [k(x+\alpha) + E (b-x) (x-a)\right]
\end{equation}
to take the form 
$$
\Psi = \mathfrak{A} x^4 + \mathfrak{B} x^3 + x^2 + \mathfrak{C}x + \frac{\mathfrak{C}\alpha}{2}
$$
required by \eqref{redux}, while still remaining positive on the interval $(a,b)$. Along the way, we must also remember to verify that  $t=x+ \alpha$ is 
strictly positive on $[a,b]$, so that  equations \eqref{explicit} and \eqref{induced} will actually  give rise to a positive definite metric $g$. 

Expanding \eqref{constrained} in powers of $x$ and  setting the 
 coefficient of $x^2$  equal to $1$, as in \eqref{redux},  we obtain the equation 
$$
\frac{k(a+b-\alpha)+ E (a^2+4ab+b^2)}{b-a}=1,
$$
which is equivalent to the requirement that 
\begin{equation}
\label{biggie}
E =\frac{k\alpha -(k+1)a -(k-1) b}{a^2+4ab+b^2}. 
\end{equation}
The other constraint imposed by requiring  that \eqref{constrained} take the form \eqref{redux} 
is that the ratio between the  constant term and the coefficient of $x$  should be 
$\alpha/2$: 
$$
\frac{\alpha}{2} = \frac{  abE-\alpha k}{\alpha k(a^{-1}+b^{-1})- (2(a+b)E+k)}
$$
This is equivalent to 
$$
  k(a^{-1}+b^{-1}) \alpha^2- (2(a+b)E-k) {\alpha} - 2  abE =0. 
$$
After  using the linear substitution \eqref{biggie} for $E$, this becomes the quadratic equation 
$$
\frac{[(a+b)\alpha + ab]\left[ (a+b)^2k\alpha + 2a^2b(k+1) + 2ab^2(k-1)\right]}{ab(a^2+4ab+b^2)}
=0
$$
 for $\alpha$. Thus, there are precisely two possible choices for $\alpha$; either 
\begin{equation}
\label{first}
\alpha = -\frac{ab}{a+b}
\end{equation}
or else 
\begin{equation}
\label{second}
\alpha = -\frac{2ab[a(k+1)+b(k-1)]}{k(a+b)^2}.
\end{equation}
Each of these choices then determines a value for $E$ via \eqref{biggie}, and thus a polynomial $\Psi (x)$ via  \eqref{constrained}. 

Of course,   $t=x+\alpha$ must be  positive on $[a,b]$ for  \eqref{explicit} to give rise to a metric on a compact manifold. 
But  this requirement is  simply  equivalent to the condition that   $a+\alpha > 0$. If \eqref{first} holds, 
$$ a + \alpha = a-\frac{ab}{a+b} = \frac{a^2}{a+b}> 0,$$
and, since $a$ and $b$ have the same sign, 
this condition is  satisfied if and only if $b>a>0$, independent of the value of $k$. By contrast, if \eqref{second} holds, we then have 
$$   \quad  a + \alpha = a  -\frac{2ab[a(k+1)+b(k-1)]}{k(a+b)^2}= - \frac{(b-a) [(k-2) ab+  ka^2]}{k (a+b)^2}, 
$$
and,  because $a$ and $b$ have the same sign,  this is positive if and only if $k=1$ and $b> a> 0$. 
From now on, we may thus  assume that $b>a>0$. 
When $k\geq 2$, we will also only need to consider the choice of $\alpha$ given by 
\eqref{first}. On the other hand, if  $k=1$,    both  \eqref{first} and \eqref{second} remain  viable candidates.

The final  condition required for \eqref{explicit} to yield  a solution on a compact manifold is that $\Psi$ must be positive on the open interval $(a,b)$. 
This will happen if and only if $(b-a) Q(x)$ is positive on the closed interval $[a,b]\subset \RR^+$. 
For the choice \eqref{first}, this
can be expressed in terms of the variable  $y= x-a$ as 
\begin{eqnarray*}
(b-a) Q(x)&=& k(x+\alpha) + E (b-x)(x-a)\\
&=& k\left( x-\frac{ab}{a+b} \right) + \frac{-\frac{kab}{a+b} -(k+1)a -(k-1) b}{a^2+4ab+b^2} (b-x)(a-x)\\
&=& k\left( y+a-\frac{ab}{a+b} \right) + \frac{-\frac{kab}{a+b} -(k+1)a -(k-1) b}{a^2+4ab+b^2} (b-a-y)y\\&=&
 \frac{ka^2}{a+b} + \frac{b^3 + (3 k-1) a b^2  + (7 k-1) a^2 b + (2k+1) a^3 }{(a + b) (a^2 + 4 a b + b^2)}y  \\ && \qquad\qquad\qquad\qquad\qquad\qquad +
\frac{(k-1) b^2 + 3 k a b  + (k+1) a^2}{(a + b) (a^2 + 4 a b + b^2)}y^2,
\end{eqnarray*}
which is strictly positive  when $y= x-a> 0$, and so is positive for $x\in [a,b]$,   for any $k\geq 1$. 
For the choice \eqref{second}, with $k=1$, we instead have
\begin{eqnarray*}
(b-a) Q(x)&=& (x+\alpha) + E (b-x)(x-a)\\
&=& \left(x-\frac{4a^2b}{(a+b)^2}\right) + \frac{-\frac{4a^2b}{(a+b)^2} -2a}{a^2+4ab+b^2} (b-x)(x-a)\\
&=& \left(y+a-\frac{4a^2b}{(a+b)^2}\right) - \frac{4a^2b+2a(a+b)^2}{(a+b)^2(a^2+4ab+b^2)} (b-a-y)y\\
&=&
\frac{a(b-a)^2 + (3a^2 + b^2) y   + 2ay^2}{(a+b)^2},
\end{eqnarray*}
which is again positive for $y=x-a\geq 0$, and so, in particular, for $x\in [a,b]$. Thus, when $k=1$, both 
\eqref{first} and \eqref{second} give us a  compact solution  for each choice of $b > a > 0$.
When $k\geq 2$, only \eqref{first}  works, but this choice in any case provides  us with a compact solution for each  $b > a > 0$.
To summarize: 

\begin{prop} 
\label{summary} 
Equations \eqref{explicit} and \eqref{induced} give rise to a strongly Hermitian Einstein-Maxwell metric $h$ on a compact
complex surface $(M,J)$  if, for some  $b> a> 0$, $\Psi$ is given by \eqref{constrained}, \eqref{biggie}, and either \eqref{first} or the $k=1$ case of \eqref{second}.  
Moreover, the construcred metrics are invariant under an action of $\mathbf{U}(2)$ on $M$ such that 
the generic orbit of $\mathbf{SU}(2)$  has fundamental group $\ZZ_k$, where $k > 0$ is the integer occurring in the expression for $\Psi$. 
\end{prop}

\section{Geometry  of the Solutions}

We have now constructed  some interesting families of Einstein-Maxwell metrics on on compact complex surfaces. It remains to completely 
understand the differential and algebraic geometry of these solutions. We begin with the following global characterization:

\begin{thm} \label{character} 
Let $h$  be a strongly Hermitian   Einstein-Maxwell metric  on a compact complex surface $(M^4,J)$
with $b_-\neq 0$.  Also suppose  that $h$   is not  a K\"ahler metric, and is  invariant under an $\mathbf{SU}(2)$-action on $M$ 
which has a  $3$-dimensional orbit. 
Then $(M,J)$ is the $k^{\rm th}$  Hirzebruch surface $\Sigma_k$  for some   $k > 0$,  and  $(M,h)$  contains an open  dense set ${\mathscr U}$ which is
isometric to a  shell $(a , b) \times (S^3/\ZZ_k)$, equipped with a  metric given by  \eqref{explicit}   and \eqref{induced}, for  $\Psi$ defined by 
\eqref{constrained}, \eqref{biggie}, and either \eqref{first} or the $k=1$ case of \eqref{second}. 
 The  set  ${\mathscr U}\subset \Sigma_k = \mathbb{P} ({\mathcal O} (k) \oplus {\mathcal O})$
is exactly   the complement of the two holomorphic 
 sections of $\Sigma_k \to \CP_1$ arising from the two sub-bundles  ${\mathcal O}(k)$ and 
${\mathcal O}$ of ${\mathcal O} (k) \oplus {\mathcal O}$, while  the restriction of  the projection $\Sigma_k \to \CP_1$ 
to ${\mathscr U}\approx ( a, b) \times (S^3/\ZZ_k)$ 
is just the composition of the factor projection  $( a,b) \times (S^3/\ZZ_k)\to (S^3/\ZZ_k)$ and 
the Hopf map $(S^3/\ZZ_k)\to \CP_1$.
\end{thm}

\begin{proof}
Recall that $h$  is said to be a strongly Hermitian Einstein-Maxwell metric on $(M,J)$ iff there is a $2$-form $F$ such that 
both $h$ and $F$ are both $J$-invariant, and such that $(h,F)$ is a solution of the Einstein-Maxwell equations (\ref{closed}--\ref{energy}). 
Since $(M,J)$ is a compact complex surface,   this is equivalent \cite[Theorem B]{lebem14}  to saying that  $h$ has constant
scalar curvature, and can be expressed as $h=f^{-2}g$, where $g$ is a K\"ahler metric on $(M,J)$, and $f\neq 0$ is a real holomorphy potential. 
In particular, this  implies that $(M,J)$ is of K\"ahler type. Since $h$ is assumed to be non-K\"ahler, $f$ must be non-constant, thereby making 
$\xi:=J\grad_g f$  a nontrivial Killing field of $g$. 
Moreover, since  $\xi f=0$, it   follows  that $\xi$ is also a Killing field of $h=f^{-2}g$.

Let $\zeta_1, \zeta_2, \zeta_3$ be  infinitesimal generators of the $\mathbf{SU}(2)$-action, where $[\zeta_1,\zeta_2]= \zeta_3$ 
and its cyclic permuations all hold. 
 Since $\mathbf{SU}(2)$ acts by isometries of $h$ which are homotopic to the identity, it preserves 
any $2$-form which is harmonic with respect to $h$, and therefore preserves the K\"ahler form $\omega$ of $g$. Consequently, it therefore preserves
the holomorphy potential $f=\pm 2^{-1/4}|\omega |^{1/2}_h$, and therefore preserves $g=f^2h$. Since the action preserves both $\omega$ and $g$, it
follows that it also preserves $J$. Thus,  the  real vector fields  $\zeta_j$ represent  infinitesimal symplectomorphisms of  $(M,\omega)$, and the 
complex  vector fields $Z_j= \eta_j- i J\eta_j$, $j=1,2,3$,  are  holomorphic vector fields on $(M,J)$. Moreover, the commutation relation
$[\zeta_1,\zeta_2]= \zeta_3$ guarantees that $\omega (\zeta_1,\zeta_2)$ is a
Hamiltonian for $\zeta_3$, and so, by taking cyclic permutations, we thus see that the $\zeta_j$ are all globally Hamiltonian vector fields. 
However, a Hamiltonian vector field is zero at a minimum of 
the Hamiltonian, and, since $M$ is compact by hypothesis, such  minima must in fact exist. This shows that 
 the Killing  fields $\zeta_j$ and the associated holomorphic vector fields $Z_j$ all   have zeros somewhere on  $M$. 

If $\mathscr{X}$ is a $3$-dimensional 
orbit of $\mathbf{SU}(2)$, then $T_pM= T_p\mathscr{X}+ J(T_p\mathscr{X})$ at any $p\in \mathscr{X}$, and  some pair of the $Z_j$   spans $T^{1,0}M$
in a neighborhood of any $p\in \mathscr{X}$; by renumbering, we may take these vector fields to be $Z_1$ and $Z_2$.
If $\alpha$ is a holomorphic $1$-form on $M$, then $\alpha$ is completely determined in a neighborhood of
$p$ by the holomorphic functions $\alpha(Z_1)$ and $\alpha (Z_2)$,   which are its components in the holomorphic co-frame dual to 
$(Z_1,Z_2)$. But $\alpha (Z_j)$ is a globally defined holomorphic function on $M$, and so is constant; and since $Z_j$ has a zero somewhere
on $M$, this constant must be zero. Thus $\alpha \equiv 0$ in a neighborhood of $p$, and therefore on all of $M$ by uniqueness of 
analytic continuation. In other words, $h^{1,0}(M)=\dim H^0(M, \Omega^1)=0$. But since  $(M,J)$ is of K\"ahler type, the Hodge decomposition 
therefore  tells us that  $b_1(M)= 2 h^{1,0}(M)=0$. In particular, $M$ has Euler characteristic $\chi (M) = 2 - 2b_1+ b_2= 2+ b_++b_-\geq 4$,
since the non-triviality K\"ahler class $[\omega]$  shows that $b_+(M)\neq 0$, and we have $b_-(M)\neq 0$ by hypothesis.

 Since   $\xi = J\grad f$  is a Killing field, the zero set of $\xi$ is a union of totally geodesic submanifolds \cite{kobfixed}; moreover,  every  normal 
 derivative of $\xi$ at such a submanifold must be non-zero,  because the restriction of $\xi$ to any normal geodesic must be a Jacobi field which is not
 identically zero. This implies that  $f$ is a generalized Morse function in the sense of Bott \cite{bottcrit}. 
 On the other hand, since 
 $\xi-iJ\xi$ is holomorphic, 
 $$\nabla_{\bar{\mu}}\nabla_{\bar{\nu}}f= g_{\lambda\bar{\nu}}\nabla_{\bar{\mu}}\nabla^{\lambda}f=0,$$
 and since $f$ is real, it therefore follows that  the Riemannian  Hessian $\nabla d f$, computed relative to $g$,  is  $J$-invariant. Consequently,  the na{\"\i}ve Hessian of 
 $f$ is  $J$-invariant at any critical point. Thus,  the critical set of $f$ is a union of totally geodesic holomorphic curves and isolated, non-degenerate critical points. 
 However, since $f$ is $\mathbf{SU}(2)$-invariant, any isolated critical point $q$ would have to be fixed by the $\mathbf{SU}(2)$-action, and since the action, being 
isometric,  commutes with the exponential map at $q$, the Hessian would also have to be invariant under a non-trivial representation of
$\mathbf{SU}(2)$ on $T_qM\cong \CC^2$, and would therefore have to be a non-zero multiple of $g$. This shows that any isolated critical point must be a non-degenerate local 
maximum or local minimum. However, the $J$-invariance of the Hessian also implies that any critical submanifold  of real dimension $2$ is also a local maximum
or minimum of $f$. In particular, there are no critical points where the Hessian has index $1$. Since $M$ is connected, the set of local minima must therefore 
be connected, because this excludes any  way to 
join  two components by passing a critical point.  Similarly, the set of local maxima must also be connected.  If  we use ${\mathcal C}_-$ and ${\mathcal C}_+$ to 
denote the sets of local minima and local maxima, respectively, it follows that each of these two sets is either a point or a compact connected Riemann surface. 
In particular, these sets have Euler characteristic $\chi ({\mathcal C}_\pm) \leq 2$, with equality iff ${\mathcal C}_\pm\cong \CP_1$. However, since $\xi$ is a Killing field, 
the Euler characteristic of $M$ coincides  \cite{kobfixed} with the Euler characteristic of its fixed point set. 
But,  since we have already observed that $\chi (M) = 2+b_+(M)+b_-(M) \geq 4$,  we therefore have  
$$4 \leq \chi (M ) = \chi ({\mathcal C}_+) + \chi ({\mathcal C}_-) \leq 2 + 2 = 4,$$
and it therefore follows   that $\chi ({\mathcal C}_\pm) = 2$, 
that  ${\mathcal C}_+\cong {\mathcal C}_-\cong \CP_1$, and that  $b_+(M) = b_-(M)=1$. 

Because $f$ is invariant under the action of $\mathbf{SU}(2)$, and because the level sets of $f$ are all connected, any $3$-dimensional orbit $\mathscr{X}$
must coincide with some   non-critical level set of $f$. However, the flow of $\grad f= -J\xi$ carries one such level set to any other, and because
the action of this flow commutes with that of $\mathbf{SU}(2)$, every  non-empty non-critical level set is conversely an $\mathbf{SU}(2)$-orbit. 
On the other hand, since the action also preserves the Riemannian distance from either critical level set ${\mathcal C}_\pm$, some, and hence  any, $3$-dimensional 
orbit $\mathscr{X}$ of $\mathbf{SU}(2)$ is diffeomorphic to the unit normal bundle of ${\mathcal C}_+$ or ${\mathcal C}_-$. In particular, the unit normal bundle of ${\mathcal C}_\pm$  has
has finite fundamental group, so the  the normal bundle of ${\mathcal C}_\pm$ is necessarily non-trivial.  Moreover, if we set $k=|\pi_1(\mathscr{X})|$ for some $3$-dimensional orbit $\mathscr{X}$,
then $k$ coincides with the absolute value $|{\mathcal C}_\pm^2|$ of the self-intersection numbers of these complex curves. However, $b_+(M)=b_-(M)=1$, and   ${\mathcal C}_-\cdot {\mathcal C}_+=0$ because 
${\mathcal C}_-\cap {\mathcal C}_+= \varnothing$; thus, only one of the curves ${\mathcal C}_+$, ${\mathcal C}_-$ can have  positive self-intersection, and only one of them can have negative self-intersection. 
At the price of possibly replacing $f$ with $-f$, we can thus arrange that ${\mathcal C}_\pm^2 = \pm k$, where $k>0$. 

Because $\mathbf{SU}(2)$ acts transitively and isometrically on each level set $\mathscr{X} = f^{-1}(t)$, the function $u = |\grad f|$ is constant on each 
level set of $f$, and $\varphi = u^{-1}df$ is therefore a closed $1$-form on the set $\mathscr{U}$ where it is defined. The unit vector field $\eta = u^{-1} \grad f$ therefore 
satisfies 
$$g_{ac}\nabla_{\eta} \eta^c = \eta^b \nabla_b \varphi_a = \eta^b\nabla_a \varphi_b = \eta^b\nabla_a g_{bc}\eta^c = \frac{1}{2} \nabla_a |\eta |^2 = 0,$$
and $\eta$ is therefore a geodesic vector field. Hence  $J\xi = - u \eta$ is  tangent to the normal  geodesic sprays of   ${\mathcal C}_+$  and ${\mathcal C}_-$.
On the other hand, we have already observed that $\xi$ corresponds, under the normal exponential maps, to a rotation vector field in the fibers of both these
normal bundles. 
If $\mathfrak{R}$  denotes the Riemannian distance from ${\mathcal C}_-$ to ${\mathcal C}_+$, 
 the Morse-theoretic picture of $f: M\to \RR$  thus amounts to saying that $M$ is the union of the normal disk bundles of radius $\mathfrak{R}/2$,
 glued together along their boundaries in such a manner that the boundary of every fiber disk is sent to the boundary of a fiber disk on the opposite side
 via a reflection. This displays $M$ as the total space of a smooth $2$-sphere bundle $\varpi: M\to S^2$.  However, 
 the fiber $2$-spheres  of the submersion $\varpi$ must be holomorphic curves, because their tangent spaces are spanned by $\xi$ and $J\xi$ on a dense set. 
Moreover, the restriction of $\varpi$  to ${\mathscr U}= M- ({\mathcal C}_-\cup {\mathcal C}_+)$
becomes  a holomorphic submersion  ${\mathscr U}\to \CP_1$ for a unique choice of complex structure on the target $S^2$, since the fibers are the orbits of a free holomorphic
$\CC^\times$-action on $\mathscr{U}$. Our submersion thus  becomes  a smooth map $\varpi : M \to \CP_1$ which is holomorphic on an open dense set, and  
and therefore holomorphic everywhere. 
Thus $(M,J)$ is the total space of a holomorphic $\CP_1$-bundle  over  a complex curve. On the other hand,   one can   show \cite[Proposition V.4.1]{bpv} that 
any such $\CP_1$-bundle is  the projectivization $\mathbb{P}({\mathscr V})$ 
 of a rank-$2$ holomorphic vector bundle ${\mathscr V}$. The two curves ${\mathcal C}_\pm$ in $M$ now determine  a pair of line sub-bundles ${\mathscr L}_\mp$ of  ${\mathscr V}\to \CP_1$ 
 such that ${\mathscr V}= {\mathscr L}_- \oplus {\mathscr L}_+$. 
 Hence $$M=\mathbb{P}({\mathscr L}_- \oplus {\mathscr L}_+) = \mathbb{P}({\mathcal O} \oplus ({\mathscr L}_-^* \otimes {\mathscr L}_+))= \mathbb{P} ({\mathcal O} \oplus {\mathcal O} (\ell))$$ for some integer $\ell$. But since  $\ell ={\mathcal C}_+^2=k$, 
 this shows  that  $(M,J)$ is biholomorphic to 
  the $k^{\rm th}$ Hirzebruch surface $\Sigma_k = \mathbb{P} ({\mathcal O} \oplus {\mathcal O} (k))$.
  
The exponential-map model shows  that $\xi$ is periodic, and generates a free circle action on ${\mathscr U}= M- ({\mathcal C}_-\cup {\mathcal C}_+)$.
Moreover, this same  model also reveals  that any  $3$-dimensional $\mathbf{SU}(2)$-orbit $\mathscr{X}$ is a non-critical level set $f^{-1}(x)$ of the Hamiltonian $f$ of $\xi$, and 
that  the circle bundle $\mathscr{X}\to \mathscr{X}/S^1$ over any 
symplectic quotient $f^{-1}(x)/S^1\approx S^2$  is isomorphic to the unit normal bundle of  ${\mathcal C}_-$, which has Chern class $-k$.  On the other hand, 
at the  price of  replacing 
$g$ with $c^2g$, for a positive  constant $c$, while simultaneously replacing $f$ with $cf$, we can now  replace $\xi$ with $c^{-1}\xi$. We can thus arrange for 
$\xi$ to have minimal period $2\pi/k$. On the universal cover $\tilde{\mathscr U}$ of $M- ({\mathcal C}_-\cup {\mathcal C}_+)$, this then implies that $\xi$ has period 
$2\pi$, and the  symplectic reduction quotient  $f^{-1}(x) \to S^2$ thus becomes the circle bundle of Chern class $-1$, with $\xi$ generating the standard
action of $S^1$. The Duistermaat-Heckman formula \cite{duiheck} thus asserts that the  area of the symplectic reduction of $f^{-1}(x)$ must therefore be $2\pi (x+\alpha)$
for some real constant $\alpha$.  Moreover, if  we adopt the convention that $\mathbf{SU}(2)$ acts from the left, the 
$\xi$ becomes a left-invariant vector field on each $\mathbf{SU}(2)$-orbit  
in $\tilde{\mathscr U}$, and so generates a right action of $\mathbf{U}(1)$ which enriches the $\mathbf{SU}(2)$-action into a $\mathbf{U}(2)$-action.
This puts now puts $g$ in the form \eqref{explicit} on $\mathscr{U}$, while simultaneously putting  $h$ in the form \eqref{induced}.
But since $g(\xi, \xi) = 2 \Psi /(x+\alpha)$ must tend to zero as we approach ${\mathcal C}_\pm$, it follows that $\Psi$ must vanish at $a= x({\mathcal C}_-)$ and 
$b= x({\mathcal C}_+)$. Moreover,  since $\xi$ has minimal period $2\pi/k$ on $M$, the derivative of $\|\xi\|$ along unit-speed geodesics orthogonal to ${\mathcal C}_\pm$ must tend
to $\mp k$ as we approach ${\mathcal C}_\pm$, so 
$$\left[\frac{2 \Psi}{x+\alpha}\right]^{1/2} \frac{d}{dx} \left[\frac{2 \Psi}{x+\alpha}\right]^{1/2}= \frac{1}{2} \frac{d}{dx}\left[\frac{2 \Psi}{x+\alpha}\right]= \left[\frac{2 \Psi}{x+\alpha}\right]^\prime \longrightarrow \pm k$$
as $x \to a^+$ or $b^-$, and hence
$$\Psi( a) = \Psi (b) =0, \quad \Psi^\prime (a) = k (a + \alpha) , \quad \Psi^\prime (b) = -k (b + \alpha).$$
Our previous  discussion   of the ODE for  $\Psi$ then shows that it must be
 defined by 
\eqref{constrained}, \eqref{biggie}, and either \eqref{first} or the $k=1$ case of \eqref{second}. 
\end{proof} 

In particular, the constructed solutions give us Einstein-Maxwell metrics on all
 the Hirzebruch surfaces $\Sigma_k = \mathbb{P}({\mathcal O}\oplus {\mathcal O}(k))\to \CP_1$ for  $k > 0$.
Of course,  this list omits  the Hirzebruch surface $\Sigma_0= \CP_1\times \CP_1$, but  $\Sigma_0$ does carry obvious  solutions provided by  cscK metrics. Our construction 
therefore proves  the following: 

\begin{thm} \label{clean} 
Let $\Sigma_k$ be any Hirzebruch surface, and let $\Omega\in H^2(\Sigma_k, \RR)$ be any K\"ahler class. 
Then $\Omega$ can be represented by a K\"ahler metric $g$ which is conformally related to an Einstein-Maxwell metric $h$. 
Moreover, if $k\geq 2$, there is exactly one such  $g$ in $\Omega$  such that $h$ is invariant under the standard action of $\mathbf{U}(2)$ on $\Sigma_k$. 
\end{thm} 
\begin{proof} If $k >0$,  we obtain such a metric on  $\Sigma_k= \mathbb{P} ({\mathcal O}\oplus {\mathcal O}(k))$  for any $b> a > 0$ by 
letting  $\alpha$  be given by \eqref{first}. 
From the symplectic point of view,  the resulting manifold is  obtained by applying symplectic cutting \cite{lerman}  to  $\RR^+\times (S^3/\ZZ_k)$, equipped with 
the symplectic form 
$$\omega = dt \wedge \sigma_3  + 2t ~\sigma_1\wedge \sigma_2$$
where the cut has been carried out 
at the level sets $t=a+\alpha$ and $t=b+\alpha$, where $t$ is a Hamiltonian for the periodic vector field $\xi$.
These level sets become the holomorphic curves ${\mathcal C}_\pm$  of self-intersection $\pm k$ arising from the two summands of 
${\mathcal O}\oplus {\mathcal O}(k)$. The symplectic form on  these curves is just $2t \sigma_1 \wedge \sigma_2$, so their areas are   given by 
$\omega ({\mathcal C}_-) = 2\pi (a+ \alpha)$ and  $\omega ({\mathcal C}_+) = 2\pi (b+ \alpha)$. 
Plugging in  the value  for $\alpha$ given by \eqref{first} therefore tells us that 
$$ \omega ({\mathcal C}_-) = \frac{2\pi a^2}{a+b}  \quad \mbox{and}\quad \omega ({\mathcal C}_+) =  \frac{2\pi b^2}{a+b}.$$
However, the fiber ${\mathcal F}$ of $\Sigma_k\to \CP_1$ is also a holomorphic curve, with  homology class  given by 
$${\mathcal F} = \frac{1}{k} \left( {\mathcal C}_+-{\mathcal C}_-\right) $$
and  ${\mathcal F}$ and ${\mathcal C}_-$ together generate $H_2 (\Sigma_k, \ZZ )$. Since 
$$\omega ({\mathcal C}_-) = \frac{2\pi a^2}{a+b}  \quad \mbox{and}\quad \omega ({\mathcal F}) =  \frac{2\pi (b-a)}{k},$$
our construction allows us to take  the areas of ${\mathcal C}_-$ and ${\mathcal F}$ to be any pair of   positive numbers by choosing
$b> a> 0$ appropriately. Since the area of any holomorphic curve is certainly positive,  every K\"ahler class on $\Sigma_k$, $k>  0$,  
is swept out  taking  $\alpha$ to be given by \eqref{first}.

On the other hand,    when $k=0$, every K\"ahler class on $\Sigma_0= \CP_1\times \CP_1$ contains a cscK metric obtained by taking the Riemannian product
of two round metrics on $S^2$ of appropriate radii. The claim therefore follows. 
\end{proof}

Theorems \ref{premier} and \ref{quatrieme} now follow, as specializations of  Theorem \ref{clean}. 

\bigskip 

Next, let us  notice that 
the family  of K\"ahler metrics $g$ on $\Sigma_1$ arising from \eqref{second}, with $k=1$, behaves quite differently from the family arising  from  \eqref{first}. 
As a matter of notation,   recall that 
 $\Sigma_1$ is also the one-point blow-up of $\CP_2$;  the curves ${\mathcal C}_-$ and ${\mathcal C}_+$ are therefore usually called ${\zap E}$ and ${\mathcal L}$, 
because  ${\mathcal L}={\mathcal C}_+$ corresponds to a generic projective {\sl line} in $\CP_2$, while the 
curve  ${\zap E}={\mathcal C}_-$ is {\sl exceptional}, in the sense that its embedding in $\Sigma_1$ is rigid. Since $\omega ({\mathcal C}_-)= 2\pi (a+\alpha)$
and $\omega ({\mathcal C}_+) = 2\pi (b+ \alpha)$,  the value of $\alpha$ provided by  \eqref{second}, with $k=1$, yields 
\begin{eqnarray*}
  \omega ({\zap E}) = \omega ({\mathcal C}_-) &=& 2\pi \left(  a  -\frac{4a^2b}{(a+b)^2}\right) = \frac{2\pi a(b-a)^2}{(b+a)^2}\\
  \omega ({\mathcal L}) = \omega ({\mathcal C}_+) &=& 2\pi \left( b  -\frac{4a^2b}{(a+b)^2} \right) = \frac{2\pi b(b+3a)(b-a)}{(a+b)^2},
\end{eqnarray*}
so that the K\"ahler metric $g$ arising  from the  data   $b> a> 0$ then  belongs to the K\"ahler class
$$\Omega = [\omega ]= \frac{2\pi b(b+3a)(b-a)}{(a+b)^2}{\mathcal L} - \frac{2\pi a(b-a)^2}{(a+b)^2}{\zap E}.$$
Writing this schematically as 
$$\Omega = u {\mathcal L} - v {\zap E},$$
we then have 
\begin{eqnarray}
\label{diff}
u-v&=& 2\pi (b-a)\\
\frac{u}{v}&=& \frac{b(b+3a)}{a(b-a)}
\label{quot}
\end{eqnarray}
and these two pieces of information of course completely determine $(u,v)$ as a function of $(a,b)$. However, this function is neither injective nor surjective. 
To clarify this point, 
 set 
$$b/a= 1+2\mathfrak{z},$$
where  $\mathfrak{z}$ is an arbitrary positive real number. Then \eqref{quot} may be rewritten as 
$$ \frac{u}{v} =   5 + 2 \left( \mathfrak{z}+ \frac{1}{\mathfrak{z}}\right).  $$
Now notice that  the right-hand side is invariant under $\mathfrak{z}\mapsto 1/\mathfrak{z}$,  tends to $+\infty$ as $\mathfrak{z}\to +\infty$, and 
has positive $\mathfrak{z}$-derivative when $\mathfrak{z} > 1$. 
It therefore follows that $u/v \geq 9$, and that any $u/v > 9$ arises from exactly two values of $\mathfrak{z}$,  which are interchanged by $\mathfrak{z} \mapsto 1/\mathfrak{z}$. 
Also note that when $\mathfrak{z}=1$, or in other words when $b/a=3$, the $k=1$ case of  \eqref{second} yields 
$$\alpha =  -\left. \frac{4a^2b}{(a+b)^2}\right|_{b/a=3} = - \frac{3}{4}$$
while plugging   $b/a=3$ into \eqref{first} similarly  results in 
$$\alpha =  -\left. \frac{ab}{a+b}\right|_{b/a=3}  = - \frac{3}{4}.$$
Plugging either of these into \eqref{biggie}, with $k=1$, thus produces exactly the same polynomial $\Psi$ for a given pair with $b= 3a > 0$. 
Thus, while we have three different solutions in a given K\"ahler class when $u/v>9$, these solutions actually merge into a single solution when $u/v=9$. This situation then 
persists on the interval  
$u/v \in (1, 9)$, where we continue to have only  one solution. To summarize: 

\begin{thm}\label{blowup}
Let $M=\Sigma_1$ be the  
 blow-up  of $\CP_2$ at a point, and let $\Omega  = u {\mathcal L} - v {\zap E}\in H^2(M, \RR)$ be a K\"ahler class. 
If $u/v\in (1,9]$, then $\Omega$ contains a unique $\mathbf{U}(2)$-invariant K\"ahler metric which is conformal to an 
 Einstein-Maxwell metric. By contrast, when $u/v\in (9, \infty )$, there are exactly three such metrics. 
\end{thm}

So far, we have been concentrating on geometric properties of the K\"ahler metric $g$, with an emphasis on its K\"ahler class. However, 
 Proposition \ref{switch} shows that different K\"ahler metrics $g$ can determine the same Einstein-Maxwell metric $h$. In the present context,
 this means that, under the orientation-reversing diffeomorphisms isotopic to the fiber-wise antipodal map of $\Sigma_k \to \CP_1$,  conformally 
 K\"ahler  Einstein-Maxwell metrics pull back to conformally K\"ahler Einstein-Maxwell metrics. Moreover, Theorem  \ref{character} guarantees that these
 pull-backs continue to belong to the constructed families. 
 
 To understand the specifics of this phenomenon, let us now calculate the 
 areas ${\mathcal A}_h$ of certain holomorphic curves with respect to the constructed Hermitian metrics $h$. 
Since $h=g/x^2$, and since the Hamiltonian $x$ is constant on ${\mathcal C}_-$ and  ${\mathcal C}_+$, we always have
$${\mathcal A}_h ({\mathcal C}_-) = \frac{2\pi (a + \alpha )}{a^2} \quad \mbox{and} \quad {\mathcal A}_h ({\mathcal C}_+) = \frac{2\pi (b + \alpha )}{b^2}.$$
iIf  $\alpha$ is given by \eqref{first}, this then implies that 
$${\mathcal A}_h ({\mathcal C}_-) = {\mathcal A}_h (\tilde{C}) =  \frac{2\pi}{a+b},$$
no matter the value of $k$. 
 By contrast, if $k=1$ and $\alpha$ is given by \eqref{second}, then 
 $${\mathcal A}_h ({\mathcal C}_-) = \frac{2\pi (b-a)^2}{a(b+a)^2}, \quad \mbox{and} \quad {\mathcal A}_h ({\mathcal C}_+) = \frac{2\pi (b+3a)(b-a)}{b(a+b)^2},$$
 and we therefore have  
$$ \frac{{\mathcal A}_h ({\mathcal C}_+)}{{\mathcal A}_h ({\mathcal C}_-)}=\frac{a^2}{b^2}\frac{u}{v} .$$
Now recall that, if we set $b/a= 1+2\mathfrak{z}$, the K\"ahler class, up to rescaling, is characterized by the number 
$$ \frac{u}{v} = \frac{b(b+3a)}{a(b-a)} = (1+2\mathfrak{z})(1+\frac{2}{\mathfrak{z}})$$
so that  $\mathfrak{z}$ and $1/\mathfrak{z}$ give rise to the same ray $\RR^+\Omega = \RR^+ ( u {\mathcal L} - v {\zap E})$ in the K\"ahler cone. 
But we now see that  
$$\frac{{\mathcal A}_h ({\mathcal C}_+)}{{\mathcal A}_h ({\mathcal C}_-)}=\frac{a^2}{b^2}\frac{u}{v}= \frac{1+2 /\mathfrak{z}}{1+2\mathfrak{z}}$$
so that interchanging these two solutions  inverts the  ratio of the areas of the holomorphic curves ${\mathcal C}_+$ and ${\mathcal C}_-$.

By Proposition \ref{switch}, all the Einstein-Maxwell metics $h$ we have constructed are ambi-K\"ahler, meaning \cite{ambikahler}  that they are conformally related 
to both a K\"ahler metric $g$ compatible with the given orientation, and with  a K\"ahler metric $\tilde{g}$ compatible with the opposite orientation. If
$\omega$ and $\tilde{\omega}$ are the K\"ahler forms of $g$ and $\tilde{g}$, then $\omega$ spans the self-dual $h$-harmonic $2$-forms on 
$M=\Sigma_k$, while $\tilde{\omega}$ spans the anti-self-dual $h$-harmonic $2$-forms. 
Because $b_2 (M)= b_2 (\Sigma_k )=2$, 
the spans $\RR [\omega ]$ and $\RR [\tilde{\omega}]$ therefore 
completely determine each other, since $[\omega ] \cdot [\tilde{\omega}]=0$ with respect to the intersection pairing. If $\mathcal F$ is the homology class
of the fiber of $\Sigma_k\to \CP_1$, the rays $\RR^+ [\omega ]$ and $\RR^+ [\tilde{\omega}]$ therefore similarly determine each other,  via the requirements
$[\omega ] \cdot [\tilde{\omega}]=0$, $[\omega ] ({\mathcal F})> 0$, and $[\tilde{\omega }]  ({\mathcal F}) < 0$. However, there is an 
orientation reversing diffeomorphism $\daleth : M\to M$ which interchanges the two relevant complex structures, and interchanges ${\mathcal C}_+$ and
 ${\mathcal C}_-$; in fact, $\daleth$  is essentially  the antipodal map on the fibers of $\varpi : M\to \CP_1$, and in particular satisfies $\daleth^2= \mbox{id}_M$. 
Now, since $\daleth$ reverses orientation, $(\daleth^* \mathbf{A}) \cdot (\daleth^* \mathbf{B}) = - \mathbf{A} \cdot \mathbf{B}$ for any $\mathbf{A}, \mathbf{B}\in H^2(M)$. 
Hence 
$$[\omega ]\cdot (\daleth^* [\omega ] ) = (\daleth^*\daleth^* [\omega ] ) \cdot (\daleth^* [\omega ] ) = - (\daleth^* [\omega ] ) \cdot  [\omega ] = - [\omega ]\cdot (\daleth^* [\omega ] )$$
and it therefore follows that $[\omega ]\cdot (\daleth^* [\omega ] ) =0$. Since we also have $(\daleth^* [\omega ] )({\mathcal F}) < 0$, it follows that 
$\RR^+ [\tilde{\omega}] = \RR^+(\daleth^* [\omega ] )$, so that $[\tilde{\omega}]$ is a positive  constant times $\daleth^* [\omega ]$. Thus,  up to scale, 
$h \mapsto \daleth^*h$ must simply permute the solutions $g$ in a given K\"ahler class. Since this permutation has square the identity, and there are an odd
number of $\mathbf{SU}(2)$-invariant solutions in any K\"ahler class, there must be a solution $g$ for which $h= \daleth^*h$.  
In light of our previous discussion, this then shows the following: 

\begin{prop}\label{swap}
Each of the  Einstein-Maxwell metrics $h$ arising from \eqref{first}  admits an orientation-reversing isometry which interchanges ${\mathcal C}_+$ and
 ${\mathcal C}_-$. On the other hand, the   two different Einstein-Maxwell metrics on $\Sigma_1$ arising via \eqref{second} from a K\"ahler class with $u/v > 9$ 
 are interchanged, up to overall scale,  by  an  orientation-reversing diffeomorphism  of $M= \Sigma_1$. 
 \end{prop}

Theorem \ref{deuxieme} is now follows from  Theorem \ref{blowup} and Proposition \ref{swap}. 

\bigskip

 
 Let  us now re-examine the scalar curvature $s_h$ of the metrics $h=x^{-2}g$. 
By equations \eqref{scale}, \eqref{redux}, and \eqref{constrained}, 
$$s_h=\kappa = - \frac{24}{\alpha} \Psi|_{x=0} = \frac{24}{\alpha} \frac{ab}{b-a}[k\alpha -ab {E}].$$
When $\alpha$ is defined by \eqref{first}, this is given by  
$$s_h=\frac{24 a b ( b^2 - a^2 + ka b)}{(b-a) (a^2 + 4 a b + b^2)}, $$ 
whereas, when $\alpha$ is given by \eqref{second}, with $k=1$, we instead have 
$$s_h= \frac{12 ab}{b-a}.$$
In particular, the constant scalar curvature of any of the  constructed metrics is necessarily positive. 

Of course, the scalar curvature is not invariant under multiplying $h$ by a constant, so it is far more interesting to instead
compute the scale-invariant quantity 
$$s_hV_h^{1/2} = \frac{\int_M s_hd\mu_h}{\int_Md\mu_h}$$
 where $V_h= \int d\mu_h$ is the total volume of $(M,h)$. 
However, since $h=x^{-2}g$,  we always have 
\begin{eqnarray*}
V_h &=& \int_{[a,b]\times (S^3/\ZZ_k)} x^{-4}~2t~ dt\wedge \sigma_1\wedge\sigma_2 \wedge \sigma_3 \\
&=&  \frac{2\pi^2}{k}\int_{x=a}^{x=b} \frac{2(x+\alpha)}{x^4}dx\\
&=& \frac{4\pi^2}{k}\left[ \frac{x^{-2}}{-2}+ \alpha \frac{x^{-3}}{-3} \right]_a^b\\
&=& \frac{4\pi^2}{k}\left[ \frac{1}{2a^2}- \frac{1}{2b^2}+ \alpha \left(\frac{1}{3a^3}-\frac{1}{3b^3}\right) \right] .
\end{eqnarray*}
If $\alpha$ is given by \eqref{first}, we therefore have 
$$V_h= \frac{2\pi^2}{k} \frac{(b-a)(a^2+4ab+b^2)}{3a^2b^2(a+b)}$$
and our previous calculation therefore tells us that 
\begin{equation}
\label{home}
s_h V_h^{1/2} = 8\pi \sqrt{6} \frac{b^2-a^2+kab}{\sqrt{k(b^2-a^2)(a^2+4ab+b^2)}}.
\end{equation}
By contrast, when $\alpha$ is given by \eqref{second}, with $k=1$, we have 
 $$
V_h= 2\pi^2 \frac{(b - a)^2 [3  b^2  + 4 a b + 5 a^2 ]}{3 a^2 b^2 (a + 
   b)^2 }
$$
and hence 
\begin{equation}
\label{hosed}
s_h V_h^{1/2} = 4\pi   \frac{\sqrt{ 6(3 b^2  + 4 a b + 5a^2) }}{(a + b)}.
\end{equation}

The curvature  of the metrics determined by \eqref{first}  thus behaves rather   differently from the curvature  of those determined by \eqref{second}. 
Indeed, when \eqref{hosed} holds, $\lim_{b/a\to \infty} s_hV_h^{1/2}$  coincides with the value of $sV^{1/2}$ for 
the standard Fubini-Study metric
on $\CP_2$.  By contrast, \eqref{home} tells us that the other families on the $\Sigma_k$ have 
   $\lim_{b/a\to \infty} s_hV_h^{1/2}$  equal to  the value of $sV^{1/2}$ 
for the standard orbifold metric on $S^4/\ZZ_k$; and  this in turn tends to $0$ as $k\to \infty$, even though  $\lim_{k\to \infty} s_hV_h^{1/2}= +\infty $   for any 
fixed value of  $b/a$. 

Let us now systematically compare the  metrics $h$ determined by \eqref{first} for a fixed cohomology class $\Omega \in H^2(M)$ on 
$M\approx S^2\times S^2$ or $\CP_2\# \overline{\CP}_2$, as we change the complex structure by letting $k$ vary. To do this,  we we will need  a  fixed 
basis for $H^2 (M)$ that is  independent of $k$. One such  basis  is provided by the classes 
\begin{eqnarray*}
{\mathcal F}&=& \frac{1}{k} \left( {\mathcal C}_+-  {\mathcal C}_-\right)\\
{\mathcal D}&=& {\mathcal C}_-  +\left\lfloor \frac{k}{2} \right\rfloor {\mathcal F}. 
\end{eqnarray*}
When $k$ is even, these then correspond to the two factors of $S^2\times S^2$; for $k$ odd, they instead correspond to the fiber class and the exceptional curve ${\zap E}$
in $\CP_2 \# \overline{\CP}_2$. Now recall that, with $\alpha$ given by \eqref{first}, the self-dual harmonic  form $\omega$ then satisfies
\begin{eqnarray*}
\omega({\mathcal F}) &=& 2\pi (b-a)\\
\omega({\mathcal C}_-) &=&\frac{2\pi a^2}{b+a}, 
\end{eqnarray*}
so that $\Omega = [\omega ]$ satisfies 
$$\frac{\Omega ({\mathcal D})}{\Omega ({\mathcal F})}= \frac{1}{(b/a)^2 -1} + \left\lfloor \frac{k}{2} \right\rfloor > \left\lfloor \frac{k}{2} \right\rfloor .$$
In particular, this shows that a given cohomology class is only  adapted to finitely many of the constructed Einstein-Maxwell metrics $h$, although this number
 grows roughly linearly as   ${\Omega ({\mathcal D})}/{\Omega ({\mathcal F})}\to \infty$. 

Now let some positive integer ${\mathbf N}\geq 2$ be given. 
Let  $\Omega \in H^2(M, \RR) $ be  the cohomology class with    ${\Omega ({\mathcal D})}= 5{\mathbf N}$ and ${\Omega ({\mathcal F})} =1$.
 Since 
 $$\frac{b}{a} = \sqrt{1+ \frac{1}{\frac{\Omega ({\mathcal D})}{\Omega ({\mathcal F})} -\left\lfloor \frac{k}{2} \right\rfloor}}$$
 we then have    $(b/a)^2 \in 
 (1 , 1+ 1/4{\mathbf N} )$  for $k= 1 , \ldots , 2{\mathbf N}$. 
   Equation \eqref{home} now  implies that 
 $$ 
   \frac{s^2_hV_h}{64\pi^2} - 
 k \left( 5{\mathbf N} - \left \lfloor \frac{k}{2} \right\rfloor \right) \in \left( 0, {\textstyle \frac{5}{4}}k+2\right) 
 $$
 and it thus follows that the values of $s_h V^{1/2}_h$  are increasing in both $k$ even and $k$ odd for $k\in \{  1, \ldots , 2{\mathbf N}\}$. This shows that, for
 $M = \CP_2\#\overline{\CP}_2$ or $S^2\times S^2$,  
  the  restriction of 
 the normalized Einstein-Hilbert functional to the Fr\'echet manifold $\mathscr{G}_\Omega(M)$ 
  has ${\mathbf N}$ different critical levels for such a choice of $\Omega$. But since Proposition \ref{comps} shows that 
  the Einstein-Hilbert functional is constant on each component of  the moduli space $\mathscr{M}_\Omega$ 
 of $\Omega$-compatible Einstein-Maxwell metrics, it follows  that $\mathscr{M}_\Omega$  has at least 
 ${\mathbf N}$ connected components.   This proves Theorem \ref{cinquieme}. 
 
  \pagebreak 
 
\section{The Page Metric Revisited} 

While various known   results  \cite{derd,lebhem,lebuniq} imply that Page's  Einstein metric \cite{page} on $\CP_2\#\overline{\CP}_2$ 
and the product metric on $\CP_1\times\CP_1$ are the 
only  conformally K\"ahler,  Einstein metrics on compact $4$-manifolds of signature zero, 
 it is still  interesting to see, in detail,  how this broad assertion manifests itself  within  the  narrower context of the present investigation. 
We will therefore wrap up  our discussion by 
concretely locating the Page metric among 
the Einstein-Maxwell metrics  constructed in this paper. 
 
 By Proposition \ref{criterion}, 
 the metric $h$ is Einstein iff $\Psi$ satisfies 
 $${\mathfrak B} = 2 \alpha {\mathfrak A}.$$
 Since we have 
 $$\Psi  = \frac{(b-x) (x-a)}{b-a}  \left [k(x+\alpha) +\frac{k\alpha -(k+1)a-(k-1)b}{a^2+4ab+b^2}(b-x) (x-a)\right]
$$
it follows that 
\begin{eqnarray*}
\mathfrak{A}&=&\frac{k\alpha  - (k+1) a - (k-1) b}{(b-a) (a^2 + 4 a b + b^2)}\\
\mathfrak{B}&=& \frac{- 2 k (a+b)  \alpha  + (k-2) b^2  + a^2 (2 + k)}{(b-a) (a^2 + 4 a b + b^2)}. 
\end{eqnarray*}
The Einstein-Maxwell metric $h$ constructed from $\Psi$ is therefore Einstein iff 
\begin{equation}
\label{looking}
2k\alpha^2 + 2  (b-a)\alpha - (k+2) a^2 - (k-2)b^2 =0.
\end{equation}

When  $\alpha$ is  given by 
  \eqref{second}, with $k=1$, equation \eqref{looking}  becomes
 $$
 2 \left( -\frac{4 a^2b}{(a+b)^2}\right)^2 - 2(b-a) \frac{4 a^2b}{(a+b)^2} - 3 a^2+ b^2 =0
 $$
 and by dividing by $a^2$ and setting $z=b/a$, this can be rewritten as 
 $$
 \frac{(z-1)^3}{(z+1)^4}\left( z^3 + 7z^2 + 13 z + 3\right) =0.
 $$
Since we automatically have $z>1$, this shows that  no Einstein-Maxwell  metric $h$ in this family is Einstein --- or  even Bach-flat.

On the other hand, if $\alpha$ is instead given by \eqref{first}, equation \eqref{looking} becomes 
$$
2k\frac{a^2b^2}{(a+b)^2}  - 2  (b-a) ab - (k+2) a^2 - (k-2)b^2 =0, 
$$
and, again dividing by $a^2$ and setting $z=b/a$, this can be rewritten as 
$$
- \frac{(k-2) z^4 + 2(k-1) z^3 + 2( k+1) z+ (k+2)}{(1+z)^2} =0.
$$
The Einstein-Maxwell metric $h$ is therefore  Einstein if and only if  $z=b/a> 1$ solves  the quartic equation 
$$(k-2) z^4 + 2(k-1) z^3 + 2( k+1) z+ (k+2)=0.$$
When $k \geq 2$, however,  the coefficients are all non-negative, so  such a solution cannot exist. 
On the other hand, when $k=1$, the equation becomes
$$z^4  - 4 z- 3=0,$$
and this actually  has a unique solution $z> 1$, because 
$z^4  - 4 z- 3$  is negative when $z=1$, has    positive derivative for $z> 1$,  and tends to $+\infty$  for large $z$. 
In fact, this solution can be expressed  in  terms of radicals by Ferrari's method, and is
explicitly given by 
 \begin{eqnarray} \label{campai} 
\frac{b}{a}= z&=& 
 \sqrt{\frac{1}{2} \left(\sqrt[3]{1 + \sqrt{2}}-\frac{1}{\sqrt[3]{1 + \sqrt{2}}}\right)} +  \\&&
 \sqrt{\left[\frac{1}{2} \left(\sqrt[3]{1 + \sqrt{2}}-\frac{1}{\sqrt[3]{1 + \sqrt{2}}}\right)\right]^{-1/2} -\frac{1}{2} \left(\sqrt[3]{1 + \sqrt{2}}-\frac{1}{\sqrt[3]{1 + \sqrt{2}}}\right)} \nonumber 
 \\&\approx& 1.784358  \nonumber \end{eqnarray}
 Since $u/v= (b/a)^2$ for the K\"ahler metrics in the family associated to the choice of $\alpha$ given by \eqref{first}, 
 Theorem \ref{pages}  therefore follows by squaring the right-hand side of   
  \eqref{campai} to obtain $u/v$.  We leave it is an exercise for the interested reader to directly compare the resulting metric $h$, as defined by \eqref{explicit}  and \eqref{induced}, 
  with the expression discovered  by Page. 
 \pagebreak 
%

\end{document}